\documentclass[preprint,10pt,authoryear]{elsarticle}

\usepackage{amssymb}
\usepackage[export]{adjustbox} 
\usepackage{amsthm} 
\newtheorem{theorem}{Theorem}
\usepackage{longtable}
\usepackage{array}  
\usepackage{multirow}
\usepackage{algorithm}
\usepackage{algorithmic}   

\usepackage{listings}

\newcolumntype{L}[1]{>{\raggedright\arraybackslash}p{#1}}
\newcolumntype{C}[1]{>{\centering\arraybackslash}p{#1}}
\newcolumntype{R}[1]{>{\raggedleft\arraybackslash}p{#1}}

\usepackage[colorlinks=true, citecolor=blue]{hyperref}
\usepackage{pdflscape}
\usepackage{rotating}
\usepackage{multirow}
\theoremstyle{definition}

\usepackage{float} 
\usepackage{booktabs}
\usepackage{amsmath}
\usepackage{array}
\usepackage{comment}

\begin{document}

\begin{frontmatter}



\title{Flow-Through Tensors: A Unified Computational Graph Architecture for Multi-Layer Transportation Network Optimization} 


\author[1]{Xuesong (Simon) Zhou}
\author[2]{Taehooie Kim}
\author[3]{Mostafa Ameli}
\author[1]{Henan (Bety) Zhu}
\author[4]{Yudai Honma}
\author[1]{Ram M. Pendyala}

\affiliation[1]{organization={School of Sustainable Engineering and the Built Environment},
    addressline={Arizona State University},
    city={Tempe},
    state={AZ},
    country={United States}}
\affiliation[2]{organization={Maricopa Association of Governments (MAG)},
    city={Phoenix},
    state={AZ},
    country={United States}}
\affiliation[3]{organization={COSYS-GRETTIA, Gustave Eiffel University},
    city={Paris},
    country={France}}
\affiliation[4]{organization={Institute of Industrial Science},
    addressline={The University of Tokyo},
    city={Tokyo},
    country={Japan}}

\begin{abstract}
Modern transportation network modeling increasingly involves the integration of diverse methodologies -- including sensor-based forecasting, reinforcement learning, classical flow optimization, and demand modeling -- that have traditionally been developed in isolation. This paper introduces Flow Through Tensors (FTT), a unified computational graph architecture that connects origin destination flows, path probabilities, and link travel times as interconnected tensors. Our framework makes three key contributions: first, it establishes a consistent mathematical structure that enables gradient-based optimization across previously separate modeling elements; second, it supports multidimensional analysis of traffic patterns over time, space, and user groups with precise quantification of system efficiency; third, it implements tensor decomposition techniques that maintain computational tractability for large scale applications. These innovations collectively enable real time control strategies, efficient coordination between multiple transportation modes and operators, and rigorous enforcement of physical network constraints. The FTT framework bridges the gap between theoretical transportation models and practical deployment needs, providing a foundation for next generation integrated mobility systems.

\end{abstract}

\begin{graphicalabstract}
    \centering
    \includegraphics[width=\linewidth]{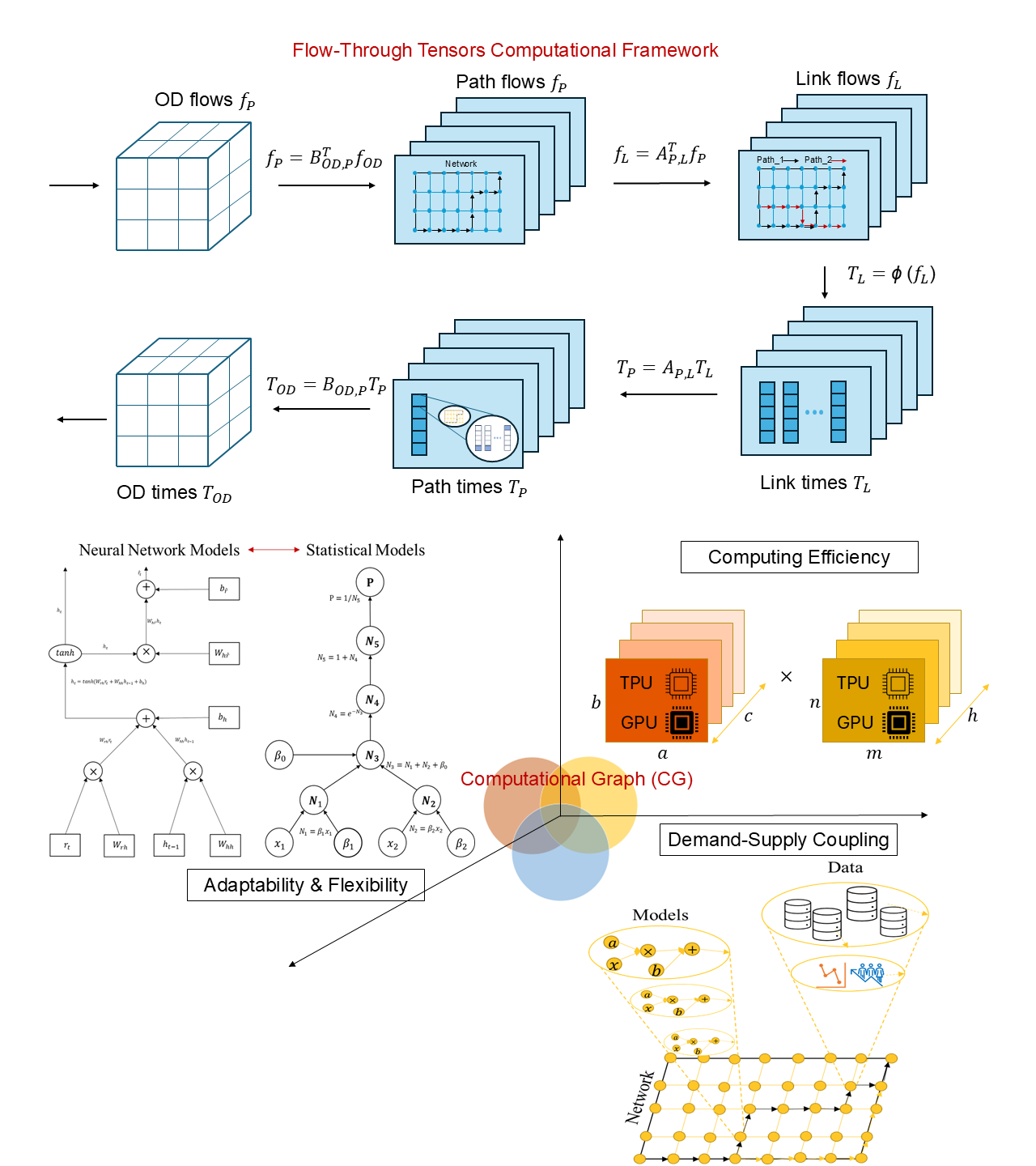}
    \label{fig:enter-label}
\end{graphicalabstract}

\begin{highlights}
\item A unified tensor-based computational graph for advanced traffic assignment across multiple modeling layers.
\item Layered KKT-based spatiotemporal modeling capturing multimodal interactions and constraints.
\item Real-time traffic management and dynamic equilibrium analysis via symbolic gradient computations.
\item Practical scalability through tensor decomposition, autodifferentiation, and ADMM-based decomposition.
\end{highlights}


\end{frontmatter}



\section{Introduction and Background: The Need for Integration in Transportation Modeling}
\label{sec:introduction}

Modern transportation networks are growing in both size and complexity. They feature multiple modes, enormous data streams, and dynamic interactions between travelers, vehicles, and infrastructure. Although significant progress has been made in short-term forecasting, local control, and classical flow modeling, many solutions remain largely disconnected and not feasible to apply to real large-scale use-cases. We see several broad traditions in current transportation systems-related research, each with distinct strengths and limitations that motivate our integrated approach:

\textbf{Sensor-Driven Traffic Prediction:} Research in this area fuses real-time or historical sensor data (loop detectors, cameras, connected vehicles) for short-term traffic state prediction \citep{yang2017origin}. Typical methods include time-series forecasting, deep learning approaches (e.g., CNN, GNN, LSTM), and multi-source data fusion. These techniques achieve high \emph{prediction accuracy} under predominantly stationary conditions but often lack the capacity for system-level modeling in scenarios characterized by strong demand–supply interactions, such as traffic incidents, severe congestion, long-term prediction and optimization, and the representation of dynamic behavioral responses. Recent advances in physics-informed deep learning have begun to address these limitations \citep{shi2021physics}.

\textbf{Reinforcement Learning \& Control:} Reinforcement learning (RL), model predictive control (MPC), and related control-theoretic frameworks are applied to local decisions such as traffic signal timing, ramp metering, or autonomous vehicle maneuvers \citep{peeta2001foundations}. Although these methods are \emph{highly adaptive} at a micro-level, extending them to network-wide equilibrium or capturing longer-term demand shifts remains a challenging and active area of research, as they are typically limited to myopic settings in multi-agent environments \citep{liu2025two}.

\textbf{Classical Network Flow Modeling:} Conventional flow models and dynamic traffic assignment (DTA) emphasize volume–delay functions, link/path capacity constraints, and user versus system objectives \citep{peeta2001foundations, ameli2019heuristic}. They offer solid theoretical foundations (user equilibrium, system optimal) but typically operate \emph{offline} and lack integration of real-time data and demand elasticity.

\textbf{Demand and Behavioral Modeling:} Activity-based, discrete choice, and tour-based models explain how travelers decide \emph{if, when, and how} to travel, considering socio-economic attributes and daily schedules \citep{rasouli2014activity}. Despite their realism, these models are computationally intensive and rarely integrated with \emph{real-time data} or AI-based supply-side interventions.

Modern transportation networks face complex challenges, including integrating multi-layer factors (origin–destination matrices, path assignments, and link flows) \citep{yang2018origin}, overcoming scalability issues in large-scale and multi-modal settings \citep{liu2019tailored}, and addressing the computational inefficiency of iterative procedures in real-world applications \citep{ameli2020simulation}. Additionally, harnessing diverse data streams for OD estimation \citep{yang2017origin} and managing multi-modal complexity (commodities, blocks, paths, trains, crew, vehicles) require advanced frameworks \citep{crainic2018intelligent, ameli2021computational}. Although solvers like CPLEX and Gurobi perform well in linear and mixed-integer programming, they lack native flow propagation, struggle with high-dimensional scalability, and offer limited support for non-linear or stochastic elements \citep{bertsimas2021interpretable}.

\subsection{Evolution of Computational Graphs in Transportation Modeling}

One potential research pathway for addressing the complexity and computational challenges in transportation system analysis is the use of new computing paradigms, such as Computational Graph (CG) and its underlying differentiable architectures, which result in graphs that describe mathematical expressions via nodes and edges. These architectures allow for the incorporation of latent domain knowledge and theory-driven approaches, while automatic differentiation facilitates gradient-based optimization of complex planning models. Consequently, backpropagation methods applied in transportation can reduce computational complexity, minimize inconsistency gaps between sub-models, and enhance convergence in interactively layered structures.

Table \ref{tab:comp_graph_evolution} shows the chronological development of the layered computational graph for travel demand estimation with forward-backward propagation.

\begin{table}[H]
\centering
\small
\caption{Representative References in Chronological Order on Computational Graphs and Forward-backward Algorithms in Transportation Network Modeling Applications; \textit{ML: Machine Learning}.\label{tab:comp_graph_evolution}}
\resizebox{\textwidth}{!}{
\begin{tabular}{L{3cm}|L{3cm}|L{4cm}}
\hline
\textbf{Reference} & \textbf{Application Domain} & \textbf{Key Contribution} \\
\hline
\cite{wu2018hierarchical} & Travel Demand Estimation & Layered computational graph with forward-backward propagation algorithm \\
\hline
\cite{ma2020estimating} & Multi-class OD Estimation & Forward-backward algorithm on computational graphs with dynamic traffic patterns \\
\hline
\cite{kim2020stepwise} & Demand Prediction & Stepwise interpretable ML for city-wide taxi demand prediction \\
\hline
\cite{kim2022computational} & Discrete Choice Models & Framework integrating econometric models such as discrete choice model with ML algorithms \\
\hline
\cite{shang2022integrated} & Transit Network Design & Layered optimization with forward-passing and backpropagation for service networks \\
\hline
\cite{lu2023physics} & Traffic State Estimation & Physics-informed neural networks on computational graphs \\
\hline
\cite{patwary2023iterative} & Bilevel Network Equilibrium Optimization & Iterative backpropagation method for efficient gradient estimation \\
\hline
\cite{liu2023end} & Equilibrium Learning & End-to-end learning of user equilibrium with implicit neural networks \\
\hline
\cite{kim2024computational} & Integrated Demand-Supply & Lagrangian relaxation-based mathematical programming reformulation \\
\hline
\cite{guarda2024estimating} & Network Flow Estimation & Day-to-day system-level data for estimating flows and travel behavior \\
\hline
\cite{liu2025end} & Equilibrium Theory & Expressivity, generalization, and optimization of user equilibrium learning \\
\hline
\cite{Du2025} & Simulation-Based Optimization & Differentiable End-to-End Simulation-Based Optimization \\
\hline
Our Work & Explicit Multi-layer Framework & Flow-Through-Tensor framework for cross-layer coordination \\
\hline
\end{tabular}
}
\end{table}

Our paper aims to synthesize existing transportation computational graph studies into a more general framework using tensor representations. In particular, we focus on problem decomposition in scenarios involving multiple blocks of scheduling and optimization, alongside modular integrations of physics-informed models and machine learning modules. While current CG methods are often applied to specific challenges—such as discrete choice model calibration, enhanced od demand estimation, traffic assignment, or transit service planning—our approach is designed to provide a unified framework that not only reviews these diverse applications but also highlights the distinctions between Flow-Through-Tensor (FTT) and existing CG methodologies. This dual focus enables us to address both broad challenges in transportation modeling and the specific limitations inherent within current CG frameworks. Building on above foundation, the subsequent sections in this paper are organized to demonstrate how each component of the framework interlocks with the others. Section 2 details the Flow-Through-Tensor framework, introducing key variables and incidence matrices that formalize the model. Section 3 focuses on tensor-based scheduling, illustrating its advantages in handling multi-period and multi-modal dynamics, while Section 4 explains the cross-block coordination achieved via advanced techniques like Alternating Direction Method of Multipliers (ADMM). Section 5 further connects our approach to modern AI methods through gradient-based optimization and backpropagation, ensuring that theoretical rigor meets practical application. Finally, Section 6 summarizes the contributions and outlines future research, reinforcing the coherent progression of ideas throughout the paper.

\section{The Flow-Through-Tensor (FTT) Framework: An Integrated Approach}
\label{sec:ftt_framework}

The key idea behind this FTT framework is to represent the transportation network using a layered structure that connects OD flows, path flows, and link flows through tensor operations and computational graphs, creating a unified modeling environment that bridges traditional transportation modeling with modern computational techniques.

Our work aims to extend the static traffic assignment and dynamic traffic assignment frameworks of \citet{mahmassani1991system} and \citet{peeta1995multiple} by introducing tensor-based representations that can efficiently handle the complexity of multi-period, multi-user class assignments. The approach balances the proactive optimization for predicted conditions with reactive adjustments to real-world deviations in supply and demand, as advocated by \citet{peeta2001foundations} in their foundational work on DTA. Furthermore, the integration of both system optimal and User Equilibrium (UE) perspectives in a temporal framework builds upon the time-dependent traffic assignment research of \citet{peeta1995system}, while the tensor formulation provides a mathematical foundation that could enhance agent-based (trip-based) modeling frameworks like POLARIS \citep{auld2016polaris} by enabling more efficient representation of complex spatio-temporal patterns.

The tensor-based approach to transportation modeling has conceptual roots in earlier information-theoretic work, particularly \citep{zhou2010information}, which explored the use of matrix determinants and traces to quantify uncertainty in OD demand estimation. In that work, OD flows were mapped to paths and then to links using link proportion matrices, establishing a mathematical foundation for the spatial and temporal mappings that we now express in higher-dimensional tensor forms. Moreover, this tensor-based structure can be naturally extended to model trip-chaining behavior by incorporating additional interaction dimensions, such as intermediate stops or sequencing constraints. For instance, Honma  \citep{honma2015spatial} proposed a spatial interaction model specifically tailored for trip chains, which conceptually parallels the current tensor-based formulation in its use of high-dimensional representations.

\subsection{Variable Definition}



In this paper, we aim to employ a consistent notation structure where scalars use lowercase italic letters (e.g., $f_{od}$, $t_\ell$), vectors (tensor forms) use bold lowercase letters (e.g., $\mathbf{f}$), matrices use bold uppercase letters (e.g., $\mathbf{A}$), and tensors use calligraphic uppercase letters (e.g., $\mathcal{F}$). Sets are represented by regular uppercase letters (e.g., $O$, $D$, $L$), while parameters typically use Greek letters. All notations in this study are presented in \ref{appendixa}.

In static traffic assignment, we consider traffic flow volumes time-independently and travel times in three different resolutions, namely origin-destination (OD) pair, path, and link. Thus, we typically track six core variables below.

 Since the primary focus in the static traffic assignment problem is solely on the dimension of traffic flow, the tensors are reduced to vectors and matrices in this section. Specifically, scalars are $0^{\text{th}}$-order tensors, vectors are $1^{\text{st}}$-order tensors, and matrices are $2^{\text{nd}}$-order tensors.

\begin{enumerate}
\item OD flow $f_{od}$ denotes the exogenous demands for OD pair $\left(o,d\right)\in OD$. 
\item OD travel time $t_{od}$ denotes the average or weighted cost for traveling from $o$ to $d$. 
\item Path flow $f_{p}$ denotes the amount of flow on route $p \in P$. 
\item Path travel time $t_{p}$ denotes the sum of the travel times on all constituent links of path $p \in P$. 
\item Link flow $f_{\ell}$ denotes the number of agents (e.g., vehicles or travelers) using link $\ell\in L$. 
\item Link travel time $t_{\ell}$ denotes the time that an agent travels through link $\ell$. Usually, $t_{\ell}$ can be viewed as a function of the link flow, especially via a volume-delay function. 
\end{enumerate}

\subsection{Two Mapping Matrices} 
We define two key matrices that map among the OD, path, and link layers:

\begin{enumerate} 
\item \textbf{Path-to-Link Incidence Matrices $\mathbf{A}_{\mathrm{P},\mathrm{L}} \in \mathbb{R} ^{\lvert P \rvert \times \lvert L\rvert}$ }\\
       Each row of $\mathbf{A}_{\mathrm{P},\mathrm{L}}$ represents a path $p$ and each column represents a link $\ell$. An arbitrary element of $\mathbf{A}_{\mathrm{P},\mathrm{L}}$, denoted by $a_{p,\ell}$, represents whether path $p$ travels through link $\ell$.

\begin{equation}
a_{p,\ell} = \begin{cases}
        1 & \text{if link } \ell \text{ is on path } p,\\
        0 & \text{otherwise}.
      \end{cases}
\end{equation}

\item \textbf{OD-to-Path Probability Matrices $\mathbf{B}_{\mathrm{OD},\mathrm{P}} \in \mathbb{R} ^{\lvert OD\rvert \times \lvert P\rvert}$}\\
 Each row of $\mathbf{B}_{\mathrm{OD},\mathrm{P}}$ represents a path $p$ and each column represents an OD pair $(o,d)$. An arbitrarily element of $\mathbf{B}_{\mathrm{OD},\mathrm{P}}$, denoted by $b_{od,p}$, represents the proportion that travelers of the OD pair $(o,d)$ will choose the path $p$.
 
\begin{equation}
b_{od,p} = \begin{cases}
        1 & \text{if path } p \text{ serves OD pair } (o,d) \text{ uniquely,}\\
        0 & \text{if path } p \text{ does not serve OD pair } (o,d),\\
        \gamma \, (\text{where } 0< \gamma <1) & \text{otherwise}.
      \end{cases}
\end{equation}

\end{enumerate}

It is important to note that the matrix $\mathbf{B}_{\mathrm{OD},\mathrm{P}}$ representing the OD-to-path mapping is not fixed, but is dynamically determined by the route choice behaviors of travelers. These behaviors may follow UE, System Optimum (SO), or be influenced by smart reservation groups, creating a dynamic mapping that evolves based on the specific equilibrium conditions and participation patterns in the network. The logit-related tensor or computational graph implementations that enable a computationally efficient implementation of our framework can be found in \cite{kim2022computational} and \cite{kim2024computational}.

A related concept is the link proportion matrix, which maps OD demand to link flow distributions. This approach has been widely used in OD matrix estimation, with foundational contributions by \citet{lo1996estimation}—who developed a statistical method based on random link choice proportions—\citet{yang1992estimation}—who addressed estimation from link traffic counts on congested networks—and \citet{cascetta1984estimation}.

\subsection{Path-Based Formulation of traffic assignment}

In the \textbf{path-based} model, the primary decision variables are the \textbf{path flows} $\mathbf{f}_P$. The standard relationships are:

\begin{enumerate}
\item \textbf{OD Flow to Path Flow}
  \begin{align}
    f_{od} &= b_{od,p}f_p , \quad \forall (o,d) \in OD.
  \end{align}

\item \textbf{Path Flow to Link Flow}
  \begin{align}
    f_\ell &= \sum_{p\in P} a_{p,\ell}\, f_p, \quad \forall \ell \in L.
  \end{align}

\item \textbf{Link Performance Function}\\
A general link performance function $\phi_{\ell}(\cdot)$ reflects the influence of the link volume $f_{\ell}$ on its travel time $t_{\ell}$, i.e., $t_\ell = \phi_\ell(f_\ell), \forall \ell \in L$. The most frequently utilized link performance function is
  \begin{align}
    t_\ell &= t_\ell^0 \left( 1 + \alpha \left(\frac{f_\ell}{C_\ell}\right)^\beta \right), \quad \forall \ell \in L,
  \end{align}
where $t_\ell^0$ is the free-flow travel time on link $\ell$, $f_\ell$ is the flow (more precisely inflow demand) on link $\ell$, $C_\ell$ is the capacity of link $\ell$, and $\alpha, \beta$ are parameters of the Bureau of Public Roads (BPR) function \citep{manual1964urban}.The parameter $\beta > 0$ serves as the congestion sensitivity parameter and is theoretically grounded in fluid-queue based volume-delay functions with polynomial arrival rates \citep{newell2013applications,cheng2022estimating, zhou2022meso}. Similarly, the parameter \(\alpha\) finds its foundation in the shape of inflow (arrival) curves, as demonstrated by Newell's fluid-queue analyses \citep{newell2013applications}. Furthermore, in \cite{zhou2022meso}, \(\alpha\) is decomposed into two distinct components: one reflecting the Demand-to-Capacity Sensitivity—which describes how quickly capacity degrades as demand approaches its limit—and another capturing the Congestion Duration Sensitivity, quantifying the impact of prolonged congestion on the longest waiting times.

\item \textbf{Path Travel Times} The travel time on path $p$ is the sum of link times along it.
  \begin{align}
    t_p &= \sum_{\ell \in L} a_{p,\ell}\,t_\ell, \quad \forall p \in P,
  \end{align}
\item \textbf{OD Travel Times}
  \begin{align}
    t_{od}&= \frac{\sum_{p\in P_{(o,d)}} f_p\,t_p}{f_{od}}, \quad \forall f_{od}>0.
  \end{align}
\end{enumerate}

By employing the aforementioned notation system, we can represent the traffic assignment problem as follows:

\begin{equation}
\underset{f_\ell}{\min} \sum_{\ell \in L} \int_0^{f_\ell} \phi_\ell(\omega)\, d\omega
\end{equation}

\subsection{Link-Based Formulation}

The traffic assignment problem can also be represented in a link-based model if we use user equilibrium as the principle for assignment, replacing the OD-to-path probability $b_{\text{od,p}}$. In the link-based (or arc-based) formulation, the link flows $\mathbf{f}_L$   serve as the primary decision variables.

Flow conservation at each node for every origin-destination (OD) pair is written as
\begin{align}
  \sum_{\ell \in \mathrm{Out}(n)} f_\ell^{od} - \sum_{\ell \in \mathrm{In}(n)} f_\ell^{od} = 
  \begin{cases}
    f_{od}, & \text{if } n = o, \\[1mm]
    -f_{od}, & \text{if } n = d, \\[1mm]
    0, & \text{otherwise},
  \end{cases}
\end{align}
where $o$ and $d$ denote the origin and destination nodes for the OD pair, respectively.
Here, $\mathrm{Out}(n)$ and $\mathrm{In}(n)$ denote the sets of outbound and inbound links at node $n$.
Summing the flows over all OD pairs then yields the total link flows $\mathbf{f}_L$.

\subsection{Flow-Through-Tensor (FTT) Formulation}

The Flow-Through-Tensor (FTT) framework offers a unified approach to transportation networks by integrating linear incidence mappings within a computational graph structure. This formulation systematically connects OD flows, path flows, link flows, and their corresponding travel times through a series of well-defined transformations.

The complete FTT system can be expressed as:
\begin{align}
\text{FTT System:} \quad
\begin{cases}
\text{Forward Flow Propagation:} \\
\quad \mathbf{f}_P = {\mathbf{B}_{\mathrm{OD},\mathrm{P}}}^{\mathsf{T}}\,\mathbf{f}_{OD} & \text{(OD flows to path flows)}\\
\quad \mathbf{f}_L = {\mathbf{A}_{\mathrm{P},\mathrm{L}}}^{\mathsf{T}}\,\mathbf{f}_P & \text{(Path flows to link flows)}\\
\quad \mathbf{t}_L = \phi(\mathbf{f}_L) & \text{(Link performance function)}\\
\text{Backward Time Propagation:} \\
\quad \mathbf{t}_P = {\mathbf{A}_{\mathrm{P},\mathrm{L}}}\,\mathbf{t}_L & \text{(Link times to path times)}\\
\quad \mathbf{t}_{OD} = {\mathbf{B}_{\mathrm{OD},\mathrm{P}}}\,\mathbf{t}_P & \text{(Path times to OD times)}
\end{cases}
\end{align}

This creates a complete computational graph connecting inputs to outputs through the composite function (as represented in Fig.\ref{fig:fft-tensor}). The objective function formulation depends on the specific transportation problem being addressed. For user equilibrium assignment, we can use either the Beckmann-McGuire-Winsten link flow-based transformation or employ a gap function based on path flows and path costs. For origin-destination matrix estimation, we utilize a generalized nonlinear least squares (NLS) formulation with OD-level variables.
\begin{equation}
\mathbf{f}_{OD} \to \mathbf{f}_P \to \mathbf{f}_L \to \mathbf{t}_L = \phi(\mathbf{f}_L) \to \mathbf{t}_P \to \mathbf{t}_{OD} 
\end{equation}

\begin{figure}[H]
    \centering
    \includegraphics[width=0.75\linewidth]{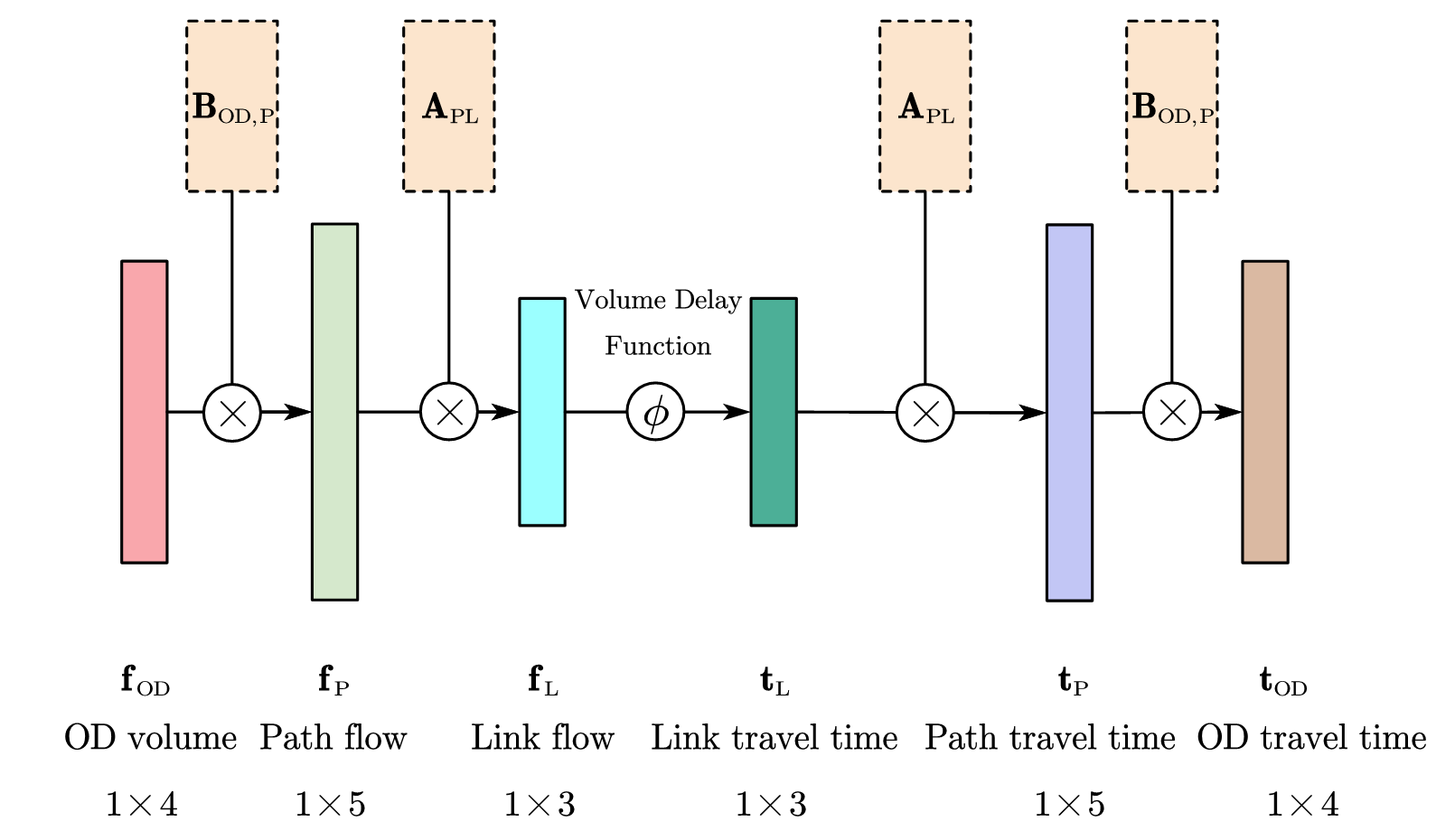}
    \caption{FTT Tensor }
    \label{fig:fft-tensor}
\end{figure}

For our FTT analysis, a fixed set of paths is assumed to be available. However, it can be dynamically adjusted through a process referred to as 'dynamic column generation' in some literature. It is important to note that the detailed mechanism for dynamic column generation is beyond the scope of this paper. Furthermore, one can incorporate a logit model within the FTT framework to update route choice probabilities $\mathbf{B}_{\mathrm{OD},\mathrm{P}}$ based on available path costs or travelers' preferences.

Fig.\ref{fig:enter-label} illustrates how the computational graph operates through an illustrative case. In this case, we know the volumes of certain links in the network, namely links a, b, and c. Those 3 critical links—selected from a total of 8 links in the example—that are strategically located to differentiate between the freeway and the frontage road. We focus on the travel demand for four OD pairs: (1, 3), (1, 4), (2, 3), and (2, 4). The paths for each OD pair are listed in Table \ref{tab:path-set-in-example}.

\begin{figure}[H]
    \centering
    \includegraphics[width=0.75\linewidth]{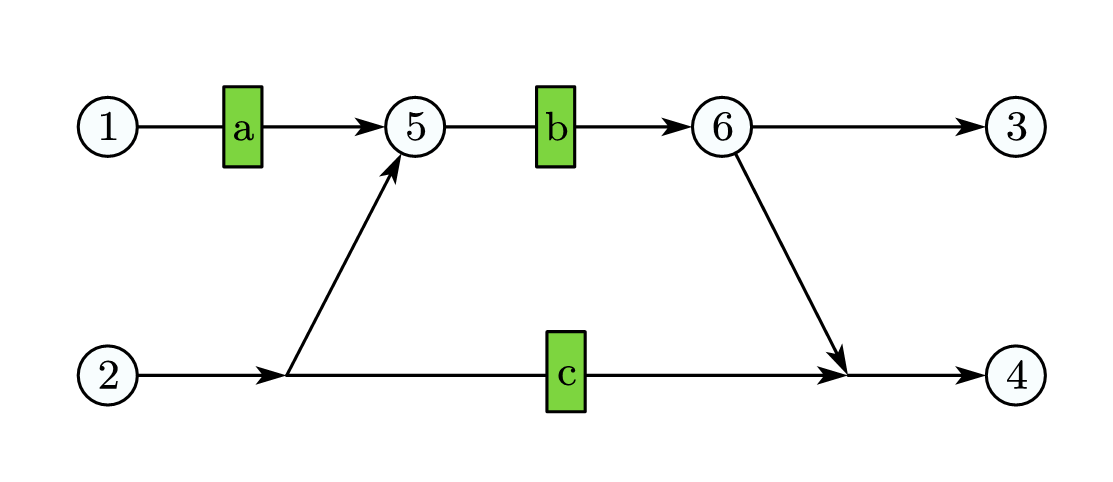}
    \caption{An Illustration Case}
    \label{fig:enter-label}
\end{figure}

\begin{table}[H]
\centering
\caption{Path Set in the Example}
\label{tab:path-set-in-example}
\begin{tabular}{llll}
\hline
\textbf{Path}     & \textbf{OD pair}& \textbf{Traversed Nodes} & \textbf{Traversed links}\\ \hline
$p_1$& (1,3) & 1-5-6-3 & a,b   \\
$p_2$& (1,4) & 1-5-6-4 & a,b   \\
$p_3$& (2,3) & 2-5-6-3 & b     \\
$p_4$& (2,4) & 2-5-6-4 & b     \\
$p_5$& (2,4) & 2-4     & c     \\ \hline
\end{tabular}
\end{table}

Path-to-link incidence matrix $\mathbf{A}_{P,L}$ in the example is represented as Table \ref{tab:path-to-link-incidence-in-example}. Thus we can write $\mathbf{A}_{P,L}$ as 

\begin{equation}
\mathbf{A}_{P,L} = \begin{pmatrix}
	1 & 1 & 0 \\
	1 & 1 & 0 \\
	0 & 1 & 0 \\
	0 & 1 & 0 \\
	0 & 0 & 1
\end{pmatrix}.
\end{equation}

\begin{table}[H]
\centering
\caption{Path-to-Link Relationship in the Example}
\label{tab:path-to-link-incidence-in-example}
\begin{tabular}{llll}
\hline
Path-Link & Link a & Link b & Link c \\ \hline
$p_1$& 1      & 1      & 0      \\
$p_2$& 1      & 1      & 0      \\
$p_3$& 0      & 1      & 0      \\
$p_4$& 0      & 1      & 0      \\
$p_5$& 0      & 0      & 1      \\ \hline
\end{tabular}
\end{table}

Assume the OD-to-path probability $b_{od,p}$ is represented as Table \ref{tab:od-to-path-probability-in-example}. Thus we can write the matrix form of $\mathbf{B}_{OD,P}$.

\begin{table}[H]
\centering
\caption{OD-to-Path Probability in the Example}
\label{tab:od-to-path-probability-in-example}
\begin{tabular}{llllll}
\hline
OD-Path & $p_1$& $p_2$& $p_3$& $p_4$& $p_5$\\ \hline
$(1,3)$    & 1      & 0      & 0   & 0      & 0      \\
$(1,4)$    & 0      & 1      & 0   & 0      & 0     \\
$(2,3)$    & 0      & 0      & 1   & 0      & 0     \\
$(2,4)$    & 0      & 0      & 0   & 0.3      & 0.7     \\  \hline
\end{tabular}
\end{table}

\begin{equation}
\mathbf{B}_{OD,P} =
\begin{pmatrix}
1 & 0 & 0 & 0   & 0 \\
0 & 1 & 0 & 0   & 0 \\
0 & 0 & 1 & 0   & 0 \\
0 & 0 & 0 & 0.3 & 0.7 \\
\end{pmatrix}
\end{equation}

Given the OD volume $\mathbf{f}_{\mathrm{OD}}=(4000, 1000, 2000, 2000)^{\mathsf{T}}$, the path flow volume $\mathbf{f}_P$ can be calculated by:

\begin{equation}
\mathbf{f}_{\mathrm{P}} =
{\mathbf{B}_{\mathrm{OD},\mathrm{P}}}^{\mathsf{T}} \mathbf{f}_{\mathrm{OD}} =
\begin{pmatrix}
1 & 0 & 0 & 0 & 0 \\
0 & 1 & 0 & 0 & 0 \\
0 & 0 & 1 & 0 & 0 \\
0 & 0 & 0 & 0.3 & 0.7 \\
\end{pmatrix}^{\mathsf{T}}
\begin{pmatrix}
4000 \\
1000 \\
2000 \\
2000 \\
\end{pmatrix}
=
\begin{pmatrix}
4000 \\
1000 \\
2000 \\
600 \\
1400 \\
\end{pmatrix}.
\end{equation}

The link flow volume can be calculated by 

\begin{equation}
\mathbf{f}_{\mathrm{L}} =
{\mathbf{A}_{\mathrm{P,L}}}^\mathsf{T} \mathbf{f}_{\mathrm{P}} =
\begin{pmatrix}
1 & 1 & 0 \\
1 & 1 & 0 \\
0 & 1 & 0 \\
0 & 1 & 0 \\
0 & 0 & 1 \\
\end{pmatrix}^{\mathsf{T}}
\begin{pmatrix}
4000 \\
1000 \\
2000 \\
600 \\
1400 \\
\end{pmatrix}
=
\begin{pmatrix}
5000 \\
7600 \\
1400 \\
\end{pmatrix}.
\end{equation}

Assuming the application of the Bureau of Public Roads (BPR) function $\phi(\mathbf{f}_L)$, we subsequently obtain $\mathbf{t}_{\mathrm{L}}=(15, 18, 10)^{\mathsf{T}}$. The next step is to calculate the path travel time $\mathbf{t}_P$ with 

\begin{equation}
\mathbf{t}_{\mathrm{P}} =
\mathbf{A}_{P,L} \mathbf{t}_L =
\begin{pmatrix}
1 & 1 & 0 \\
1 & 1 & 0 \\
0 & 1 & 0 \\
0 & 1 & 0 \\
0 & 0 & 1 \\
\end{pmatrix}
\begin{pmatrix}
15 \\
18 \\
10 \\
\end{pmatrix}
=
\begin{pmatrix}
33 \\
33 \\
18 \\
18 \\
10 \\
\end{pmatrix}.
\end{equation}

The OD travel time $\mathbf{t}_{OD}$ can be calculated by: 
\begin{equation}
\mathbf{t}_{\mathrm{OD}} =
\mathbf{B}_{OD,P} \mathbf{t}_P =
\begin{pmatrix}
1 & 0 & 0 & 0   & 0 \\
0 & 1 & 0 & 0   & 0 \\
0 & 0 & 1 & 0   & 0 \\
0 & 0 & 0 & 0.3 & 0.7 \\
\end{pmatrix}
\begin{pmatrix}
33 \\
33 \\
18 \\
18 \\
10 \\
\end{pmatrix}
=
\begin{pmatrix}
33 \\
33 \\
18 \\
12.4 \\
\end{pmatrix}.
\end{equation}

\subsection{Gradient-Based Network Automatic Differentiation}
Our approach leverages the explicit gradient relationship between OD flows and travel times. We compute the gradient $\frac{\partial \mathbf{t}_{OD}}{\partial \mathbf{f}_{OD}}$ through each mapping in the Flow-Through-Tensor framework:

\begin{align}
\frac{\partial \mathbf{f}_P}{\partial \mathbf{f}_{OD}} &= {\mathbf{B}_{\mathrm{OD},\mathrm{P}}}^{\mathsf{T}} &\text{(OD flows to path flows)}\\
\frac{\partial \mathbf{f}_L}{\partial \mathbf{f}_P} &= {\mathbf{A}_{\mathrm{P},\mathrm{L}}}^{\mathsf{T}} &\text{(Path flows to link flows)}\\
\frac{\partial \mathbf{t}_L}{\partial \mathbf{f}_L} &= \text{diag}\left(t_{\ell}^0 \, \alpha \, \beta \, \frac{1}{C_\ell} \left(\frac{f_{\ell}}{C_\ell}\right)^{\beta-1}\right) &\text{(Link flows to link times)} \label{eq:diag-link-performance} \\
\frac{\partial \mathbf{t}_P}{\partial \mathbf{t}_L} &= \mathbf{A}_{\mathrm{P},\mathrm{L}} &\text{(Link times to path times)}\\
\frac{\partial \mathbf{t}_{OD}}{\partial \mathbf{t}_P} &= \mathbf{B}_{\mathrm{OD},\mathrm{P}} &\text{(Path times to OD times)}
\end{align}

In Eq.(\ref{eq:diag-link-performance}) the matrix is a square diagonal matrix. This is because we assume symmetric cost functions in our traffic assignment problem, which naturally leads to a diagonal Jacobian where each element corresponds to an individual link. Discussions regarding asymmetric cost functions are beyond the scope of this paper.

By the chain rule, the overall sensitivity, which can be used in the OD demand estimation setting, is:
\begin{align}
\frac{\partial \mathbf{t}_{OD}}{\partial \mathbf{f}_{OD}} = \mathbf{B}_{\mathrm{OD},\mathrm{P}} \cdot \mathbf{A}_{\mathrm{P},\mathrm{L}} \cdot \frac{\partial \mathbf{t}_L}{\partial \mathbf{f}_L} \cdot {\mathbf{A}_{\mathrm{P},\mathrm{L}}}^{\mathsf{T}} \cdot {\mathbf{B}_{\mathrm{OD},\mathrm{P}}}^{\mathsf{T}}
\end{align}

The computational graph structure is particularly advantageous for this type of transportation modeling. By representing the network as a differentiable computational graph, we can leverage automatic differentiation to efficiently compute these gradients. This formulation ensures physical consistency with traffic flow theory and provides explicit causal understanding of how demand changes propagate through the network to affect travel times via congestion, representing a fundamental improvement over purely correlational models that may fail to identify underlying mechanisms.

\subsection{Comparison of Computation}

The table below summarizes \textbf{how each approach} (path-based, link-based, FTT) handles the six variables. FTT leverages the same definitions but encodes them as mappings ($\mathbf{B}$ and $\mathbf{A}$), making the entire problem a chain of composite functions:
\begin{table}[H]
\centering
\caption{Comparison of Link-Based, Path-Based, and Flow-Through-Tensor Formulations}
\label{tab:comparison}
\begin{tabular}{p{2.5cm} p{3cm} p{3cm} p{3cm}}
\hline
\textbf{Aspect} & \textbf{Link-Based} & \textbf{Path-Based} & \textbf{FTT} \\
\hline
\textbf{Primary Variables} 
 & Link flows \((f_\ell)\). All-or-nothing (AON) subproblem defines auxiliary solution. 
 & Path flows \((f_p)\). Multiple paths per OD, often updated via reduced gradient. 
 & Layered transformations: \(\mathbf{f}_P \to \mathbf{f}_L \to \mathbf{t}_L \to \mathbf{t}_P\). Maintains \emph{computational graph}. \\
\hline
\textbf{Exact Solution Algorithm}
 & Frank--Wolfe combines current and AON flows via line search. Slow convergence near optimum. 
 & Gradient Projection updates each path flow along negative-gradient. Faster convergence. Manages path set. 
 & Compatible with gradient-based/quasi-Newton solvers. Suitable for GPU/parallel computation. \\
\hline
\textbf{Path Management}
 & No explicit path enumeration. Link flows are main storage.
 & Explicit, growing set of active paths. New shortest paths added as needed. 
 & Column-generation or fixed path set. Emphasis on \emph{computational graph} across layers. \\
\hline
\textbf{Line Search}
 & Merges current link flows and AON flows each iteration.
 & Often direct step in negative-gradient at path level, sometimes with path-flow line search.
 & Compatible with various step-size routines (ADAM, L-BFGS) using chain-rule derivatives. \\
\hline
\textbf{Gradient Computation}
 & Link-level cost linearization each iteration.
 & Reduced gradients for path flows. Shifts flow from higher-cost to lower-cost paths.
 & Layered forward-backward pass applies chain rule. Simplifies gradient evaluation, extends to new cost layers. \\
\hline
\textbf{Features}
 & Minimal storage; simpler coding.\newline
   Path details require extra steps; computationally infeasible for large-scale application.
 & Direct path flow control; faster convergence; easier path results.\newline
   Higher memory for large networks.
 & Compatible with auto-differentiation; supports complex costs.\newline
   May still need path generation for many paths. \\
\hline
\end{tabular}
\end{table}

\subsection{Embedded Optimization and Differentiable Modeling Perspectives}

The FTT structure shares core mathematical principles with neural networks and optimal control, involving layered mappings with forward information flow and backward gradient propagation. These connections are formalized through Lagrangian relaxation and Karush-Kuhn-Tucker (KKT) conditions, where adjoint variables serve a role analogous to backpropagation in deep learning. A detailed treatment, including chain rule derivations and applications to transportation networks, is provided in~\ref{appendixb}. This analysis reveals the fundamental connection between neural network backpropagation, the method of adjoints in optimal control, and the Flow-Through-Tensor framework in transportation networks. All three approaches use Lagrangian duality with a forward pass to compute states and a backward pass to propagate sensitivities via adjoint variables. This unified perspective enables cross-pollination between deep learning advances and transportation network optimization through their common mathematical foundation.

While our work focuses on transportation-specific optimization problems, it draws conceptual alignment with recent advancements in differentiable and embedded optimization. In particular, our approach is consistent with developments that integrate optimization layers within machine learning architectures.

Specifically, our work builds upon the backward propagation framework across different layers of variables introduced in \citet{wu2018hierarchical}, which independently adopted a layered structure conceptually aligned with approaches such as OptNet \citep{amos2017optnet} and differentiable convex optimization layers \citep{agrawal2019diffopt}. These developments, along with advances in program rewriting for convex formulations \citep{agrawal2018rewriting} and end-to-end constrained optimization learning \citep{kotary2021endtoend}, provide foundational insights into embedding optimization problems directly within neural network architectures.

In our case, we embed the transportation system’s internal mappings as a composite sequence of functions: from path flows to link flows, from link flows to travel times, and from travel times back to path costs. This structured formulation ensures both physical interpretability and computational consistency. By embedding these mappings directly into the optimization framework, we enable the model to capture flow propagation dynamics, travel time estimation, and cost evaluation in a cohesive and data-efficient manner.

 In our framework, the transportation system is modeled as a series of interdependent mappings—linking path flows, link flows, travel times, and finally back to path costs. This setup inherently creates computational cycles, akin to a “chicken and egg” scenario, where determining the demand influences the supply and vice versa. These cycles require iterative fixed-point methods to balance the system, where the solution stabilizes once the interactions between demand and supply converge to a fixed point. 

To efficiently resolve these cycles, we leverage a variable-splitting reformulation approach \citep{fisher1997vehicle, niu2018coordinating, kim2024computational}, which dualizes the coupled constraints. This reformulation breaks cyclic dependencies and restores a directed acyclic graph (DAG) structure at the subproblem level, enabling modular decomposition. Furthermore, for more complex and non-differentiable constraints—such as capacity limitations or intermodal competition—we adopt Lagrangian relaxation techniques.

It should be remarked that, while tensors do admit low-rank decompositions, such as CANDECOMP/PARAFAC (CP), Tucker, tensor-train (TT), and tensor-ring (TR) decompositions, these representations introduce non-convex optimization landscapes. This is analogous to matrix factorization, where although matrix completion may appear convex when posed directly over the full matrix X, it becomes non-convex when formulated in terms of its low-rank factors A and B. In both cases, convexity depends on how the optimization problem is posed. We also note that tensor completion is, in general, non-convex and NP-hard. However, tensor methods are specifically designed for modeling high-dimensional, multi-way data. In many practical applications, such data exhibits an underlying low-rank tensor structure. Using tensor decompositions in such settings helps preserve the intrinsic multi-dimensional relationships within the data, which can be critical for performance and interpretability.

\section{Motivation for Multi-Dimensional, Tensor-Based Scheduling}

Transportation network models traditionally focus on single-period or single-day equilibrium analysis, wherein each traveler selects a route to minimize their immediate travel time. While mathematically tractable, these static equilibrium formulations often result in inefficient outcomes due to the use of time-independent cost calculation and decentralized decision-making---encapsulated by the well-known \textit{Price of Anarchy} \citep{KOUTSOUPIAS200965}. Real-world transportation systems, however, are inherently multi-dimensional: demand fluctuates by day and time, user behaviors vary across different travel contexts, and capacity constraints evolve dynamically.

To address these inefficiencies, recent literature identifies a range of coordination strategies. \citet{roughgarden2005selfish} formalized the inefficiencies of selfish routing, motivating mechanisms for improving system performance. Extensions such as bounded rationality \citep{di2014braess} and the impacts of ride-sourcing \citep{beojone2021inefficiency} highlight the real-world complexity of uncoordinated behavior. Reservation-based methods, including tradable mobility credits \citep{yang2011managing,nie2012transaction} and intersection scheduling \citep{dresner2008multiagent}, demonstrate promising demand management solutions. Coordination through equilibrium-based frameworks also shows potential: \citet{papadopoulos2021personalized} propose Pareto-improving freight routing, \citet{ning2023robust} develop a correlated equilibrium routing mechanism, while \citet{luan2024passenger} and \citet{zhang2025bounding} explore rerouting and automation-based scheduling in transit and multi-modal systems. These methods underscore the value of integrating dynamic, multi-dimensional coordination into traffic modeling frameworks.

A tensor-based representation naturally accommodates this complexity. Beyond the basic dimensions of "User Group × Day × Time Period," we can encode a much richer set of attributes within a single multi-dimensional structure including:

\begin{itemize}
    \item \textbf{User Classes:} Different traveler segments (e.g., single-occupant vehicles, high-occupancy vehicles, freight), value of time, utilities, and costs.
    \item \textbf{Vehicle Types:} Distinctions between electric, diesel, and other vehicle technologies
    \item \textbf{Mobility Chain:} Sequences of trips within a day (e.g., home-work-home, home-school-shopping-home)
    \item \textbf{Trip Purposes:} Commute, education, shopping, leisure, and other activity types
    \item \textbf{Time-of-Day Variations:} Morning peak, midday off-peak, evening peak, or finer granularity
    \item \textbf{Day-of-Week Effects:} Weekday versus weekend travel differences
\end{itemize}

With these tensor formulations, one can analytically compare how various multi-period or multi-day rotation strategies improve upon the single-period equilibrium baseline. Crucially, such models enable explicit calculation of cost reductions and Pareto improvements, revealing how partial scheduling or "smart reservations", in which only some portion of travelers coordinate their route or time-of-day choices, yield measurable efficiency gains. These gains scale with the congestion sensitivity of the network (captured by parameters like $\beta$ in volume-delay functions) and the fraction of users participating in the scheduling scheme. Therefore, the proposed methodology can effectively capture the heterogeneity and dynamics of the system's demand and supply.

\subsection{Mathematical Formulation Using Pigou's Network}

We analyze a classic Pigou network with two parallel routes: route a with constant travel time $t_a = 1$, and route b with variable travel time $t_b = x^{\beta}$, where $x$ is the flow on route b. With total demand normalized to 1 unit and $\beta = 1$, the system optimal (SO) solution divides flow equally between routes ($t_{SO} = \frac{3}{4}$), while user equilibrium (UE) concentrates all flow on route $b$ ($t_{UE} = 1$), yielding a Price of Anarchy $\text{PoA} = \frac{t_{UE}}{t_{SO}} = \frac{4}{3}$ that quantifies the inefficiency of uncoordinated routing decisions.
\subsubsection{Multi-Period Tensor Formulation}

To extend the single-period analysis to a multi-period setting, we introduce a rotation mechanism represented by a binary tensor $\mathbf{R} \in \{0,1\}^{I \times D}$, where $I$ denotes the number of user groups and $D$ the number of time periods (e.g., days). The entry $R_{i,d} = 1$ indicates that group $i$ follows the SO assignment on day $d$, while $R_{i,d} = 0$ corresponds to UE behavior.

To illustrate the rotation strategy, we consider a two-day planning horizon. The participating users are divided evenly into two subgroups. Subgroup 1 ($P1$) is assigned to route $a$ on Day 1 and switches to route $b$ on Day 2. Subgroup 2 ($P2$) follows the reverse pattern, using route $b$ on Day 1 and route $a$ on Day 2. The non-participants (NP), who follow UE behavior, consistently choose route $b$ on both days and do not engage in the rotation. 

An illustrative example of the rotation tensor is provided in Table~\ref{tab:rotation-matrix}, where two user groups alternate between SO and UE assignments over a 2-day horizon to promote temporal fairness and load balancing.

\begin{table}[H]
\centering
\caption{Example of Rotation Tensor}
\renewcommand{\arraystretch}{1.2}
\begin{tabular}{>{\bfseries}c|ccccc}
\toprule
\textbf{Group} $\boldsymbol{\backslash}$ \textbf{Day} & 1 & 2 & 3 & 4 \\
\midrule
$R_{1,d}$ & 1 & 0 & 1 & 0 \\
$R_{2,d}$ & 0 & 1 & 0 & 1 \\
\bottomrule
\end{tabular}
\label{tab:rotation-matrix}
\end{table}

The resulting temporal flow pattern for day $d$ is computed as:

\begin{equation}
x^{(d)} = \sum_{i=1}^{I} \left( R_{i,d} \cdot x^{\text{SO}}_i + (1 - R_{i,d}) \cdot x^{\text{UE}}_i \right)
\end{equation}

Here, $x^{\text{SO}}_i$ and $x^{\text{UE}}_i$ denote the flow contributions of group $i$ under system-optimal and user equilibrium behaviors, respectively.

This formulation enables temporal coordination across user groups while preserving long-run fairness and reducing overall system cost as well as calculation cost.

A participation rate $p$ denotes the participation group—users who agree to rotate their route choices across periods. The remaining fraction $1 - p$ consists of non-participants.

Each participating subgroup alternates between route a, which has a constant travel time, and route b, whose travel time depends on the total flow. Table~\ref{tab:rotation-summary} presents each group's route choices, flow assignment, and travel time over two days. 

\begin{table}[H]
\centering
\caption{Flow and Travel Time Tensor over Two Days}
\renewcommand{\arraystretch}{1.4}
\begin{tabular}{>{\bfseries}l >{\bfseries}l c c c}
\toprule
\multicolumn{2}{c}{\textbf{}} & \textbf{P1} & \textbf{P2} & \textbf{NP} \\
\midrule
\multirow{3}{*}{Day 1} & Route       & a & b & b \\
                      & Flow        & $f_a = \frac{p}{2}$ & $f_b = 1 - \frac{p}{2}$ & $f_b = 1 - \frac{p}{2}$ \\
                      & Travel Time & $t_a = 1$ & $t_b = \left(1 - \frac{p}{2} \right)^{\beta}$ & $t_b = \left(1 - \frac{p}{2} \right)^{\beta}$ \\
\midrule
\multirow{3}{*}{Day 2} & Route       & b & a & b \\
                      & Flow        & $f_b = 1 - \frac{p}{2}$ & $f_a = \frac{p}{2}$ & $f_b = 1 - \frac{p}{2}$ \\
                      & Travel Time & $t_b = \left(1 - \frac{p}{2} \right)^{\beta}$ & $t_a = 1$ & $t_b = \left(1 - \frac{p}{2} \right)^{\beta}$ \\
\bottomrule
\end{tabular}

\label{tab:rotation-summary}
\end{table}


\begin{figure}[H]
    \centering
    \includegraphics[width=0.75\linewidth]{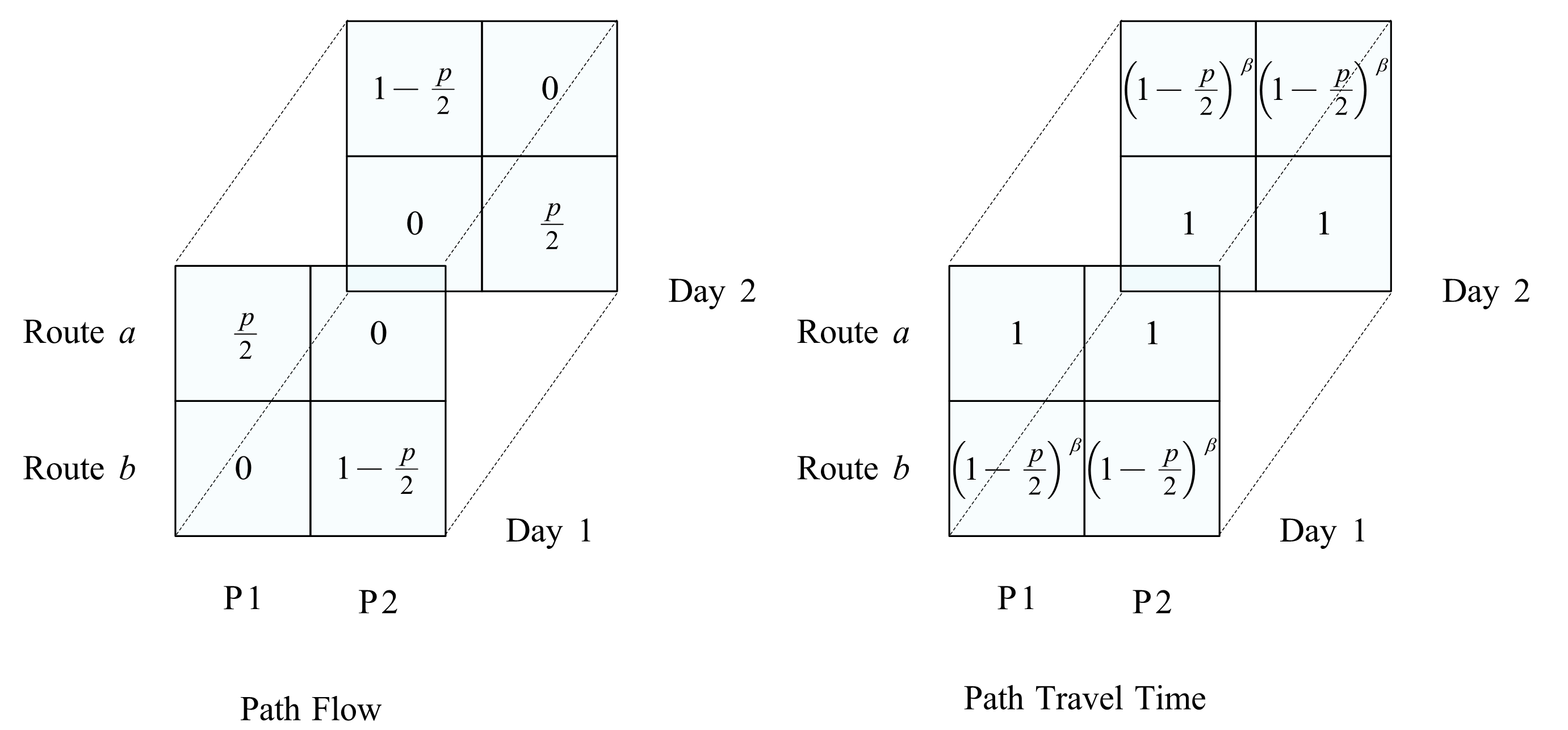}
    \caption{$3^\text{rd}$-order flow and travel time tensor}
    \label{fig:flow-and-travel-time-tensor-over-two-days}
\end{figure}

For participants, we compute the average travel time (a "mode‐average" over the day dimension):
\begin{align}
    \overline{t}_{P} &= \frac{1 + (1 - p/2)^\beta}{2} \quad
\end{align}

For non‐participants, it is simply:
\begin{align}
    \overline{t}_{NP} &= (1 - p/2)^\beta \quad \label{eq:tnp}
\end{align}

Theorem~\ref{thm:pareto} formalizes the key result that coordination through the proposed rotation mechanism leads to Pareto-improving outcomes. This simple strategy achieves two important goals: (1) stable network capacity utilization across days, and (2) improved travel times for all users compared to pure user equilibrium. The detailed proof of this result is provided in ~\ref{appendixc}.

\begin{theorem}
\label{thm:pareto}
The tensor-based rotation with participation rate $p > 0$ and rotation $\alpha = 1/2$ yields a Pareto improvement over pure User Equilibrium for all $\beta \geq 1$, corresponds to decreasing $PoA$.
\end{theorem}

We note that while our framework illustrates benchmark improvements under the rotation/smart reservation scenario, the specific methods for achieving optimized route or scheduling assignments for limited road resources—such as through mobility credits or full trajectory scheduling of automated vehicles—are beyond the scope of this paper. Our goal is to provide clear benchmarks and insights, leaving detailed implementations of these approaches for future research.

\subsection{Multi-Dimensional Tensor Representation and Management Strategies}

It is important to clarify that while we introduce a multi-dimensional tensor representation to capture the rich dynamics of transportation systems, we do not provide detailed implementations of these management strategies in this paper. Our discussion here is intended as a conceptual exploration---an invitation for further research---rather than a fully developed operational framework.

In practice, many metropolitan areas, such as Beijing with its day-of-week license plate restrictions, already employ rudimentary forms of multi-dimensional traffic management. Similarly, more dynamic approaches---such as automated vehicle scheduling---could benefit from the tensor-based mechanisms we propose. Our aim is to promote the use of tensors in transportation modeling even as we recognize that the detailed design of such strategies remains an open area for future work.

We begin with a generalized multi-dimensional tensor representation:
\begin{equation}
\mathcal{X}(\text{class}, \text{vehicle}, \text{tour}, \text{purpose}, \text{tod}, \text{dow}, \text{route}, \text{day}),
\label{eq:tensor}
\end{equation}
which captures a rich set of attributes beyond traditional dimensions.

Tensor decomposition methods allow us to extract latent travel patterns from this high-dimensional data. For example, using CP decomposition, we express $\mathcal{X}$ as:
\begin{equation}
\mathcal{X} \approx \sum_{r=1}^{R} \lambda_r \cdot \mathbf{a}^{(1)}_r \circ \mathbf{a}^{(2)}_r \circ \cdots \circ \mathbf{a}^{(N)}_r,
\label{eq:cp}
\end{equation}
where each component $r$ corresponds to a distinct travel cluster and $\lambda_r$ indicates its relative prominence. In a weekday morning peak scenario, CP decomposition might reveal:
\begin{equation}
\begin{aligned}
\lambda_1 &= 0.50 \quad &\text{(Work commuters in single-occupancy vehicles)},\\[1mm]
\lambda_2 &= 0.30 \quad &\text{(School drop-offs)},\\[1mm]
\lambda_3 &= 0.20 \quad &\text{(Medical or personal appointments)}.
\end{aligned}
\label{eq:lambda}
\end{equation}

Based on these latent components, we introduce a factor-guided rotation tensor $\mathbf{R}$ that tailors interventions according to the identified patterns:
\begin{equation}
\mathbf{R}_{i,d} = 
\begin{cases}
    \text{Pattern A} & \text{if user } i \text{ aligns with Component 1},\\[1mm]
    \text{Pattern B} & \text{if user } i \text{ aligns with Component 2},\\[1mm]
    \text{Pattern C} & \text{if user } i \text{ aligns with Component 3}.
\end{cases}
\label{eq:rotation_tensor}
\end{equation}

To design a practical scheduling mechanism, we extend our rotation scheme over a longer consultation period---say, a five-day workweek or even a multi-week cycle (e.g., four weeks per month). Suppose the overall participation rate is \(p = 0.10\) (10\%), but note that this rate may vary by latent component. For example, work commuters (Component 1) may have less flexibility (\(p_1 < 0.10\)), while school drop-offs (Component 2) could be more amenable to adopting system-optimal (SO) assignments (\(p_2 \approx 0.10\)).

We then define a day-specific flow for each component using a rotation schedule. Let \(R_{r,d}\) indicate whether the system-optimal assignment is active for component \(r\) on day \(d\). The aggregated flow on day \(d\) is given by:
\begin{equation}
x^{(d)} = \sum_{r=1}^{R} \left( R_{r,d} \cdot x^{\text{SO}}_{r} + \left(1 - R_{r,d}\right) \cdot x^{\text{UE}}_{r} \right),
\label{eq:day_flow}
\end{equation}
where \(x^{\text{SO}}_{r}\) and \(x^{\text{UE}}_{r}\) represent the flows under system-optimal and user equilibrium conditions for component \(r\), respectively.

For a concrete example, consider a four-week cycle with five working days each week (\(D = 20\)). For school drop-offs (Component 2), we might define the rotation schedule as:
\begin{equation}
R_{2,d} = 
\begin{cases}
1, & \text{if } d \in \{1, 6, 11, 16\}, \\
0, & \text{otherwise}.
\end{cases}
\label{eq:rotation_schedule}
\end{equation}
This means that on the first day of each week, the designated 10\% of school drop-offs adopt the system-optimal assignment---even if it means a non-preferred arrival time---while the remaining days follow the user equilibrium behavior. Different rotation schedules can be devised for other components based on their inherent flexibility.

For additional flexibility, Tucker decomposition provides a complementary approach:
\begin{equation}
\mathcal{X} \approx \mathcal{G} \times_1 \mathbf{A}^{(1)} \times_2 \mathbf{A}^{(2)} \cdots \times_N \mathbf{A}^{(N)},
\label{eq:tucker}
\end{equation}
where the core tensor \(\mathcal{G}\) captures higher-order interactions across dimensions, such as the joint influence of vehicle type and time-of-day on congestion.

In summary, while our primary contributions focus on tensor-based scheduling for network coordination, the multi-dimensional tensor framework and its associated management strategies presented here are intended as conceptual benchmarks. We recognize that the detailed mechanisms for extracting actionable insights and implementing these strategies---especially over extended time horizons with differential participation rates---are beyond the scope of this paper. Nonetheless, we believe that these ideas demonstrate the potential of tensor-based approaches and encourage further research in this promising direction.

\section{ADMM Framework for Modular Integration in Transportation Systems}

\subsection{Background and Applications in Transportation Modeling}
Modern transportation systems combine modular components—each with internal logic, decision variables, and objectives. Examples include data-driven OD estimation, physics-informed flow propagation (e.g., Flow-Through Tensor), and vehicle assignment systems. To preserve modularity while enforcing cross-module consistency, we deploy the Alternating Direction Method of Multipliers (ADMM), first introduced by \citet{gabay1976dual}.

The ADMM-based decomposition approach has been widely applied to transportation problems, including traffic assignment \citep{kim2024computational,LIU2023103233, ZHANG2024novel}, dynamic optimal transport \citep{CHI2024Improved}, and optimizing intercity express transportation networks \citep{DONG2025Gradient}. Earlier applications include vehicle routing problems with time windows \citep{yao2019admm}, cyclic train timetabling \citep{zhang2019solving}, and integrated vehicle assignment and routing for shared mobility \citep{liu2020integrated}. Recent work by \citet{chen2025enhancing} has also applied similar approaches to enhancing high-speed railway timetable resilience.

\subsection{Mathematical Formulation and Algorithm}
The ADMM approach enables decomposable optimization while enforcing physical and behavioral constraints. The general mathematical formulation is as follows:
\begin{align}
\min_{\mathbf{x}, \mathbf{z}} \quad & \mathcal{J}_1(\mathbf{x}) + \mathcal{J}_2(\mathbf{z}) \\
\text{subject to} \quad & \mathbf{C}\mathbf{x}_S + \mathbf{D}\mathbf{z}_S = \mathbf{b},\\[6pt]
& \mathbf{x} = [\mathbf{x}_S, \mathbf{x}_{NS}], \\[4pt]
& \mathbf{z} = [\mathbf{z}_S, \mathbf{z}_{NS}],
\end{align}
where:
\begin{itemize}
\item \textbf{Decision Variables:}
\begin{itemize}
\item $\mathbf{x} = [\mathbf{x}_S, \mathbf{x}_{NS}]$: Complete decision variables in Module 1, explicitly partitioned into shared ($\mathbf{x}_S$) and non-shared ($\mathbf{x}_{NS}$) subsets.
\item $\mathbf{z} = [\mathbf{z}_S, \mathbf{z}_{NS}]$: Complete decision variables in Module 2, similarly partitioned.
\end{itemize}
\item \textbf{Objective Functions:}
\begin{itemize}
\item $\mathcal{J}_1(\mathbf{x})$, $\mathcal{J}_2(\mathbf{z})$: Module-specific objectives, dependent on their respective full sets of variables.
\end{itemize}
\item \textbf{Coupling Constraint:}
\begin{itemize}
\item Matrices $\mathbf{C}$ and $\mathbf{D}$ and vector $\mathbf{b}$ enforce consistency explicitly through the shared variables ($\mathbf{x}_S$, $\mathbf{z}_S$).
\end{itemize}
\end{itemize}
\textbf{Remark 1:} The bent variables refer to how variables are handled across different subproblems or components of a larger optimization problem, especially in decentralized or distributed optimization settings. Here, the coupling variable sets $\mathbf{x}_S$ and $\mathbf{z}_S$ represent interrelated quantities such as link-level vehicle flow volume versus passenger link flow volume. These coupling values can be expressed as link performance functions. The complete variable sets $\mathbf{x}$ and $\mathbf{z}$ include flow-through tensor variables across OD, path, and link layers.

Coupling constraints in transportation networks serve as critical interconnection mechanisms between passenger flows and vehicle operations, forming the backbone of what can be conceptualized as a flow-through tensor framework. These constraints can be mathematically formulated through various approaches, including Lagrangian relaxation techniques that dualize the coupling relationships, thereby enabling decomposition into more tractable subproblems \citep{palomar2006tutorial, chiang2007layering}. The coupling functions themselves represent sophisticated demand-supply interactions that extend beyond traditional Bureau of Public Roads (BPR) formulations of volume-to-capacity ratios. More nuanced formulations include those that capture passenger-vehicle capacity matching \citep{niu2018coordinating}, where constraints such as $F_{\text{passenger}} \leq S \cdot F_{\text{vehicle}}$ ensure that passenger flows remain within the bounds of available vehicle capacity. Additionally, multi-commodity flow approaches treat passengers and vehicles as distinct commodities sharing network resources, with constraints that limit their weighted sum on each link \citep{fisher1997vehicle}.

In multi-modal systems, this two-block tensor structure accommodates:
\begin{enumerate}
\item \textbf{Temporal Synchronization Coupling}, where the time dimensions of passenger and vehicle tensors must align to minimize transfer times \citep{ceder2001creating}, with frequency coordination constraints like $f_i = k \cdot f_j$ linking the temporal patterns across tensor blocks \citep{nesheli2014optimal};
\item \textbf{Spatial Resource Allocation Coupling}, where the spatial dimensions of both tensor blocks share limited infrastructure resources, expressed as $\sum_m (\alpha_m \cdot F_m) \leq C$, with $\alpha_m$ representing how each tensor block consumes shared capacity \citep{yang2004multi}.
\end{enumerate}

Coupling Function can also be modeled as a Performance Function for Demand-Supply Interaction
The typical forms include: (1) Ratio-based forms ($D/C$): As in the BPR function; (2) Difference-based forms (D-C): These measure the gap between demand and supply capacity; (3) Queue-based formulations: Performance functions can explicitly capture dynamic queueing phenomena, including vertical and horizontal queue propagation, spillback effects, and time-dependent delays, providing more realistic congestion representation in scenarios with time-varying demand patterns. Our current implementation supports analytical approximations of queueing behavior but does not yet include full mesoscopic simulation such as DTALite \citep{zhou2014dtalite}. However, the modular architecture of the tensor framework allows for future incorporation of more detailed traffic loading schemes, including simulation-based approaches.

These coupling mechanisms ensure that solutions optimized for each tensor block individually will maintain feasibility and efficiency when operating as an integrated system. Moreover, the framework’s elegance lies in its dual approach: ML captures demand variations, while FTT captures both demand and supply distributions.
This modularity aligns with the FTT framework's layered mappings---e.g., from $\mathcal{F}_{OD} \rightarrow \mathcal{F}_P \rightarrow \mathcal{F}_L \rightarrow \mathcal{T}_L \rightarrow \mathcal{T}_{OD}$---enabling seamless coordination between independently updated modules.

The ADMM algorithm proceeds iteratively through three steps:
\begin{align}
\mathbf{x}^{k+1} &:= \arg\min_{\mathbf{x}} \mathcal{J}_1(\mathbf{x}) + (\boldsymbol{\lambda}^k)^\top (\mathbf{C} \mathbf{x}_S + \mathbf{D} \mathbf{z}_S^k - \mathbf{b}) + \frac{\rho}{2} \|\mathbf{C} \mathbf{x}_S + \mathbf{D} \mathbf{z}_S^k - \mathbf{b}\|_2^2 \\
\mathbf{z}^{k+1} &:= \arg\min_{\mathbf{z}} \mathcal{J}_2(\mathbf{z}) + (\boldsymbol{\lambda}^k)^\top (\mathbf{C} \mathbf{x}_S^{k+1} + \mathbf{D} \mathbf{z}_S- \mathbf{b}) + \frac{\rho}{2} \|\mathbf{C} \mathbf{x}_S^{k+1} + \mathbf{D} \mathbf{z}_S - \mathbf{b}\|_2^2 \\
\boldsymbol{\lambda}^{k+1} &:= \boldsymbol{\lambda}^k + \rho \left( \mathbf{C} \mathbf{x}_S^{k+1} + \mathbf{D} \mathbf{z}_S^{k+1} - \mathbf{b} \right)
\end{align}where $\rho > 0$ is the penalty parameter and $\boldsymbol{\lambda}$ is the dual variable.
This setup ensures:
\begin{itemize}
\item Each subproblem updates its \textbf{entire variable set} $\mathbf{x}$ or $\mathbf{z}$.
\item Only the \textbf{bent components} $\mathbf{x}_S$, $\mathbf{z}_S$ are used in the coupling term, which keep consistency across the calculation while decouple the framework structure. 
\end{itemize}

\textbf{Remark 2:} The multiplier minimization and alternating process is naturally suitable for GPU-based computational acceleration in deep learning applications. This alternating structure facilitates efficient parallel computation on modern hardware architectures. ~\cite {Kim} proposed a sample implementation of this approach. For further details, readers may refer to \cite{kim2024computational}, while the problem decomposition and Lagrangian-based iterative process requires explicit coding.

\textbf{Remark 3:} While enabling ADMM on parallel computing platforms is technically emerging areas of research, it is beyond the scope of this paper. Interested readers can explore related work such as \cite{LIU2023103233}

\textbf{Remark 4:} Regarding an important question of whether flow models can be extended to traditional applications with strong 0-1 constraints (e.g., resource-constrained shortest path problems): such extensions typically require a two-phase approach. First, continuous flow variables should be treated as a continuous relaxation problem with floating-point variables. Subsequently, these solutions can be converted to integer or binary variables through projection or rounding techniques, employing specialized algorithms such as branch-and-bound or cutting plane methods. While the integration of continuous and discrete mappings lies outside the scope of this paper, interested readers can refer to \cite{yao2019admm, liu2020integrated} for comprehensive treatments of these techniques.

\subsection{Transportation Applications of ADMM: Passenger-Vehicle Coupling}

Transportation systems involve coordination between passenger flows and vehicle movements. We decompose the system into two primary blocks:

\paragraph{Passenger Block}
The Passenger Block manages flows and travel times through flow tensors ($\mathcal{F}^\text{Passenger}_{OD}$, $\mathcal{F}^\text{Passenger}_{P}$, and $\mathcal{F}^\text{Passenger}_{L}$) and their corresponding travel time tensors ($\mathcal{T}^\text{Passenger}_{OD}$, $\mathcal{T}^\text{Passenger}_{P}$, and $\mathcal{T}^\text{Passenger}_{L}$). 

\paragraph{Vehicle Block}
The Vehicle Block tracks vehicle flows, travel times, and costs through flow tensors ($\mathcal{F}^\text{Vehicle}_{OD}$, $\mathcal{F}^\text{Vehicle}_{P}$, $\mathcal{F}^\text{Vehicle}_{L}$), travel time tensors ($\mathcal{T}^\text{Vehicle}_{OD}$, $\mathcal{T}^\text{Vehicle}_{P}$, $\mathcal{T}^\text{Vehicle}_{L}$), and cost tensors ($\mathcal{C}^\text{Vehicle}_{P}$, $\mathcal{C}^\text{Vehicle}_{L}$). 

\begin{figure}[H]
    \centering
    \includegraphics[width=1\linewidth]{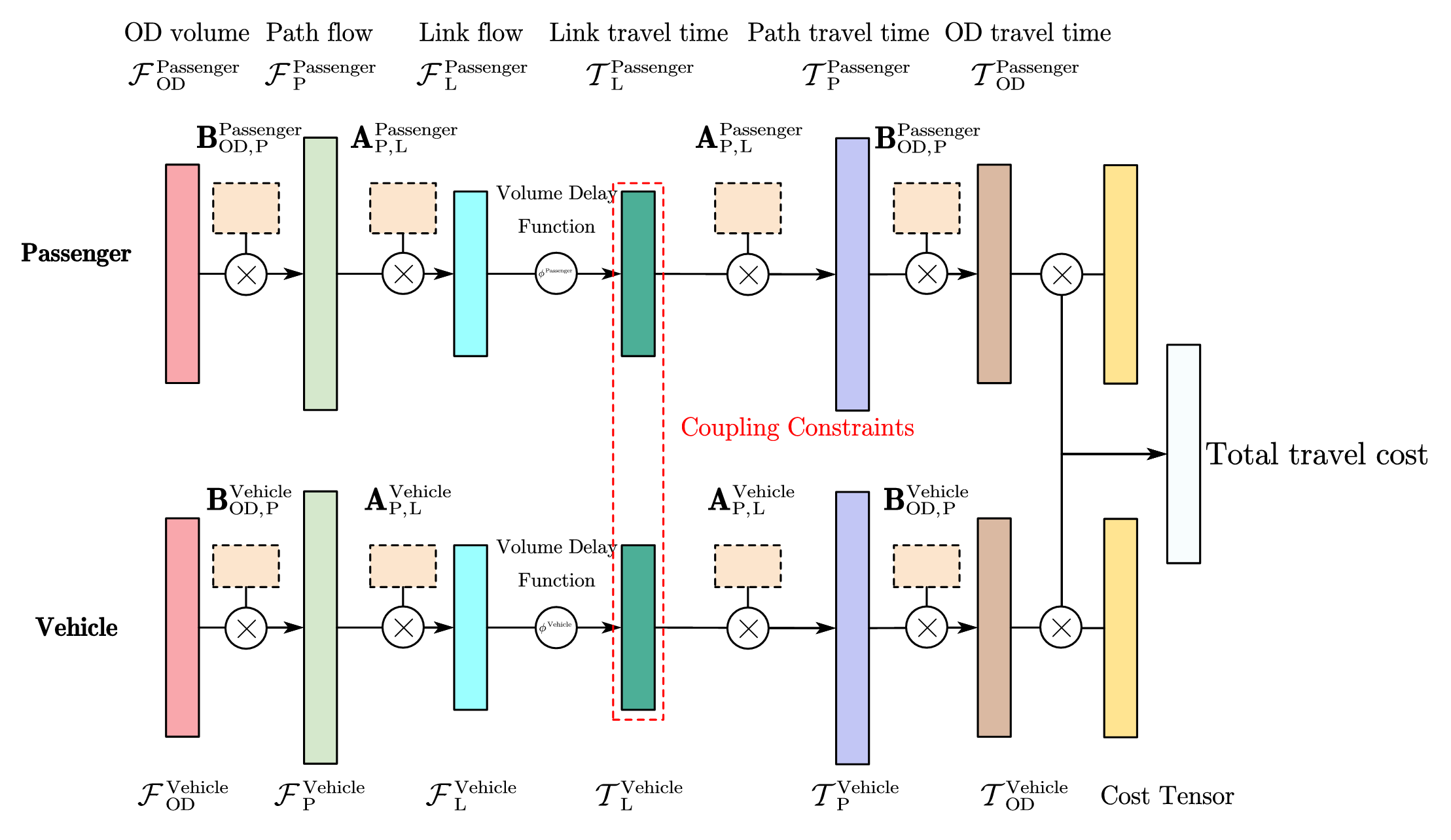}
    \caption{Coupling of Passenger and Vehicle Tensors}
    \label{fig:coupling-of-passenger-and-vehicle-tensors}

\end{figure}

The coupling constraint enforces vehicle capacity limitations:
\begin{equation}
\mathcal{F}^\text{Passenger}_{L} \leq \boldsymbol{\omega} \odot \mathcal{F}^\text{Vehicle}_{L} \quad
\end{equation}

where $\boldsymbol{\omega}$ is the vector of occupancy rates (passengers per vehicle) for each link/mode, $\odot$ denotes element-wise multiplication.

Each block optimizes its own objective:
\begin{align}
\text{Passenger Objective:} \quad & \mathcal{J}_1(\mathbf{x}) = \sum_{od} \mathcal{F}^\text{Passenger}_{OD} \cdot \mathcal{T}^\text{Passenger}_{OD,P} \\
\text{Vehicle Objective:} \quad & \mathcal{J}_2(\mathbf{z}) = \sum_{\ell} \mathcal{F}^\text{Vehicle}_{L} \cdot \mathcal{C}^\text{Vehicle}_{L}
\end{align}

The ADMM framework allows each system to optimize independently while ensuring the physical constraint that passengers cannot exceed vehicle capacities. Classical ADMM ensures convergence under convex and separable conditions. For more complex, nonconvex scenarios, convergence analysis is beyond the scope of this paper. Our formulation employs block-wise updates with regularization to support practical implementation.

\subsection{Integrated ML-FTT Framework}
To illustrate the potential for integration between the FTT framework and various ML approaches, Table \ref{tab:ml_ftt_integration} summarizes different ML methods that can be integrated with FTT elements through the ADMM framework. This integration leverages the strengths of both approaches, the physical consistency and interpretability of FTT, combined with the data-driven insights from ML models, with the underlying optimization structure and backpropagation formulation detailed in ~\ref{appendixb}.

\begin{table}[htbp]
\centering
\caption{ML Methods for Transportation Network Elements and FTT Integration; NNs: Neural Networks, GNNs: Graph Neural Networks, DL: Deep Learning, GP: Gaussian Process.}
\renewcommand{\arraystretch}{1.1}
\resizebox{\textwidth}{!}{%
\begin{tabular}{>{\bfseries}p{2cm} p{2.2cm} p{6.7cm} p{2cm}}
\toprule
\textbf{Modeling Element} & \textbf{ML Methods} & \textbf{Research Focus} & \textbf{FTT Application} \\
\midrule
OD Demand/Flow/Pattern & GCN; LSTM; Transformer; Bayesian Learning & Spatio-temporal and residual networks modeled OD patterns \citep{zhang2017deep,chu2019deep,xu2024space}, extended by graph-based methods \citep{wang2019origin,shi2020predicting,zheng2022short}, transformers and ML for inflow prediction \citep{hu2023high,hu2023examining}, simulation with Bayesian optimization \citep{huo2023simulation}, and heterogeneous graphs for direct-transfer flows \citep{tang2024origin}. & $\mathcal{F}_{OD}$ tensor with $\mathbf{B}$ mapping matrix \\
\hline
Path/Route Flow & Deep RL; GNNs; Multi-task Learning & Multi-task deep models integrated user-POI features \citep{huang2020multi}, later enhanced by environment-aware RL for adaptive mobility \citep{guo2021route}. & $\mathcal{F}_{P}$ enables differentiable path selection \\
\hline
Link/Cell Flow & Spatio-temporal GNNs; Tensor Models; Gaussian Processes & Uncertainty captured via entropy \citep{liu2022analysis}, link capacity estimated using BPR-informed DL \citep{huo2022quantify}, improved GPs for prediction and scalability \citep{liu2022gaussian,liu2023gaussian}, reinforced dynamic GCNs for imputation and prediction \citep{chen2022novel}, data fusion for missing states \citep{xing2022traffic}, and transfer learning for low-coverage networks \citep{zhang2023full}. & $\mathcal{F}_{L}$ with $\mathbf{A}_P$ preserves flow conservation \\
\hline
Link/Cell Travel Time and Traffic State & Graph Learning; Ensemble Models; Tensor Decomposition & Spatial-temporal dependencies modeled via GNNs \citep{zhao2019t,guo2019attention}, adaptive structures without predefined graphs \citep{bai2020adaptive}, GPS-based speed estimation via matrix completion \citep{yu2020urban}, grid-based traffic state prediction using deep spatio-temporal models \citep{liu2021deeptsp}, GPs, DL, entropy, and fusion techniques for scalable link-level inference \citep{zhang2023full,liu2022gaussian,liu2023gaussian,huo2022quantify,liu2022analysis,xing2022traffic}. & $\mathcal{T}_{L}$ maps flows to travel times via gradient matrices \\
\hline
Path/Route Travel Time & Graph NNs; Tensor-based NN; Deep Sequence Models & Modeled using segmented ML \citep{bahuleyan2017arterial}, deep spatio-temporal networks \citep{wang2018will,shen2020ttpnet}, real-time bus data \citep{ma2019bus}, production-level GNNs \citep{derrow2021eta} and physics-informed Transformers for trajectory estimation and prediction \citep{geng2023physics, long2024physics}. & $\mathcal{T}_{P}$ with $\mathbf{A}^{\top}\mathcal{T}_{L}$ aggregates link times \\
\hline
OD/Trip Travel Time & IRL; Transformer; Oracle Learning & Addressed via ConvLSTM hybrids \citep{duan2019prediction}, diffusion-transformer frameworks \citep{lin2023origin}, and personalized IRL-based models \citep{liu2025personalized}. & $\mathcal{T}_{OD}$ with $\mathbf{B}^{\top}\mathcal{T}_{P}$ \\
\bottomrule
\end{tabular}%
}
\label{tab:ml_ftt_integration}
\end{table}

In transportation networks, ADMM can be used to enable effective coordination between machine learning-based OD estimation and flow-through tensor-based traffic assignment. This coupling creates a bridge between data-driven approaches and physics-based modeling principles.

As a conceptual discussion, we present the key concepts below:

\begin{itemize}
    \item ML modules that estimate OD flows from available data.
    
    \item FTT modules that handle traffic assignment using these flows.
    
    \item Coupling constraints that ensure consistency between these two perspectives.
\end{itemize}

The ML component minimizes the difference between estimated and observed OD flows while incorporating appropriate regularization. Simultaneously, the FTT component optimizes traffic assignment while respecting physical network constraints and observed travel times.

For deep learning models, such as convolutional neural networks (CNNs) used in OD estimation, the ADMM framework modifies the loss function to include penalty terms that enforce consistency with the FTT-derived flows. This approach preserves the strengths of both paradigms—the data adaptation capabilities of ML and the physical interpretability of FTT.

By integrating ADMM with the FTT framework and modern deep learning approaches, we create an architecture that respects physical constraints while allowing each module to evolve and optimize independently. This framework is flexible, scalable, and generalizable to other coordination problems in transportation systems, including pricing, demand shaping, multimodal integration, and disruption management.
\section{Implementation Details: Open-Source Implementation Resources}
\label{sec:implementation}
The FTT framework has been implemented through a series of open-source repositories that address different aspects of transportation modeling. These implementations provide concrete evidence of the framework's computational efficiency and scalability for large-scale networks.

The tensor-based computational architecture has been realized through several complementary GitHub repositories, organized into three functional categories:

\subsubsection*{Foundation Components}
\begin{itemize}
    \item \textbf{CG Network Model}~\citep{CGNetworkModel}: Core network modeling components of the FTT framework.
    
    \item \textbf{CG Choice Model}~\citep{CGChoiceModel}: Computational graph-based discrete choice models for efficient handling of discrete choice probabilities.
\end{itemize}

\subsubsection*{Integration Mechanisms}
\begin{itemize}
    \item \textbf{CG Equilibrium}~\citep{Kim}: Mathematical integration of travel demand and traffic network models through ADMM techniques.
    
    \item \textbf{Traffic State Estimation}~\citep{Lu}: Implementation of traffic state estimation using computational graphs.
\end{itemize}

\subsubsection*{Applications and Planning Tools}
\begin{itemize}
    \item \textbf{CG-Based Transportation Planning Models}~\citep{CGTPM}: Open-source Python codes for digitizing traditional transportation planning models through computational graphs.
    
    \item \textbf{TCGlite}~\citep{TCGlite}: Big data driven computational graph solutions for education and academic research.
\end{itemize}

\begin{figure}[htbp]
    \centering
    \includegraphics[
        trim=8cm 0cm 3cm 0cm,
        clip,
        width=1.3\textwidth]{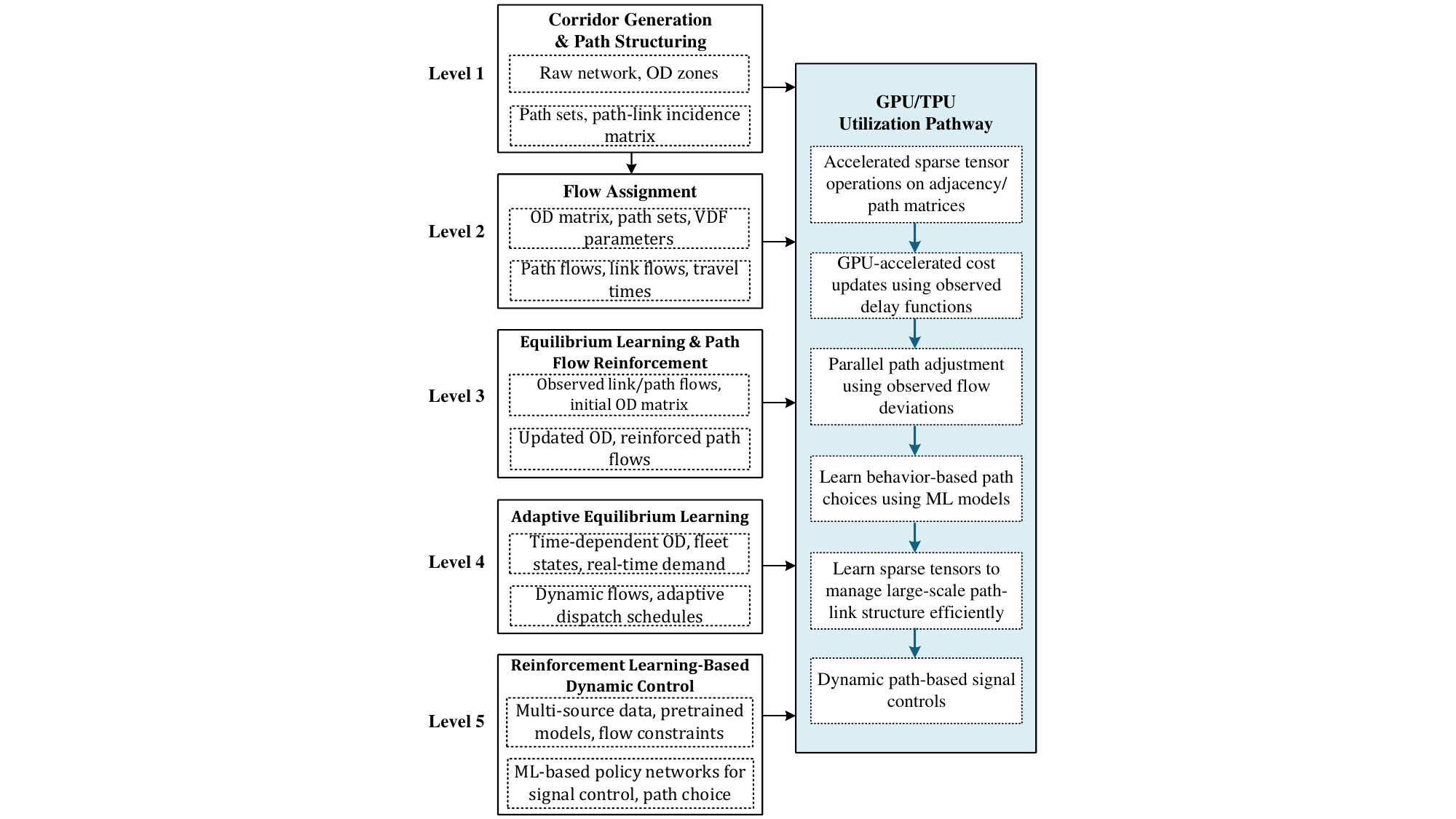}
    \caption{Modular Pipeline of the Tensor-based Flow Regulation Framework}
    \label{fig:tensor_pipeline}
\end{figure}

Figure~\ref{fig:tensor_pipeline} presents a modular pipeline outlining the tensor-based architecture, offering a unified framework view. It illustrates a level-wise progression from corridor generation to ML-driven dynamic control, with each stage corresponding to key modules such as flow assignment and equilibrium learning, supported by the open-source tools described above. The figure connects the system’s software components with its theoretical and computational foundations. While our framework incorporates traffic flow rules and behavior models, it is not a full physics informed reinforcement learning approach. Rather, it serves as a model-based environment that enables structured learning grounded in observed traffic dynamics and system constraints.

All implementations follow the General Modeling Network Specification (GMNS) standards \citep{lu2023virtual}, with networks generated from OpenStreetMap data using the \texttt{osm2gmns} tool \citep{osm2gmns}. This standardization ensures interoperability across different modules and enables seamless application to real-world transportation networks. The framework incorporates several advanced algorithmic techniques for computational efficiency. Column generation principles, as described in \citet{liu2020integrated}, are leveraged for efficient path-based solution approaches, supported by the Path4GMNS 
library \citep{Path4GMNS} which provides specialized algorithms for shortest path finding and column generation in GMNS-formatted networks. For rapid evaluation and calibration of large-scale networks, the framework integrates with DTALite, a queue-based mesoscopic traffic simulator available as a PyPI package \citep{DTALite}, making it easily accessible for incorporation with the tensor-based framework. Additionally, the implementation employs cross-block coordination through ADMM to support efficient multi-modal and multi-operator optimization, with a practical implementation available in the AVRLite repository \citep{AVRLite} that provides specialized algorithms for autonomous vehicle routing with decomposition techniques.

Tensor-based implementations using TensorFlow or PyTorch have proven remarkably effective in estimating discrete choice models for traffic network analysis, particularly by leveraging automated differentiation within computational graphs. This approach not only provides reliable numerical accuracy for gradient-based methods but also demonstrates significant gains in handling large behavioral datasets. In a multinomial logit model with 88 parameters reported by \citet{CGChoiceModel}, the tensor-based approach completed in just 8.6 seconds compared to alternative methods requiring 7-9 minutes. More impressively, a nested logit model with 89 parameters converged in about 30 seconds, outperforming other approaches that took over 15 minutes. These improvements stem from using automatic differentiation within computational graphs, while alternative approaches use chain rule or numerical derivative techniques such as Richardson extrapolation—all with the same log-likelihood objective function and BFGS optimization algorithm.
Beyond choice modeling, the framework excels in traffic assignment scenarios. In the study by \citep{guarda2024estimating}, a case study on a Fresno network with 1,078 nodes, 6,000 OD pairs, and an 18,000-path set achieved convergence in approximately 600 seconds over 100 episodes using tensor-based implementation with automatic differentiation, resulting in a relative gap within a percent. Meanwhile, research applying end-to-end learning of user equilibrium with implicit neural networks \citep{liu2023end} to a Chicago sketch network (933 nodes) reported CPU running times between 1,000 and 4,000 seconds.
It is worth noting, however, that current implementations in the above studies have not fully explored GPU utilization, suggesting room for further speed gains by leveraging parallel hardware acceleration. Additionally, various configurations of initial values, path flows, and other model inputs can significantly influence performance. This suggests that hybrid approaches—such as using efficient lower-level C++ implementations for OD flows and path assignment alongside tensor-based methods—may be equally viable for certain applications. While Python libraries like TensorFlow already utilize lower-level C++ libraries internally, purely tensor-based methods may not always be the optimal choice. Combining or customizing C++ modules for baseline flows with computational graph frameworks can maintain robustness while preserving the computational efficiency benefits demonstrated in large-scale network applications.

To improve GPU utilization, one could use sparsity patterns in origin-destination by path by time tensors to enable efficient batched kernels, and apply mixed precision arithmetic in tensor decomposition steps to reduce computation time and memory use.
These improvements can help fully leverage parallel hardware and are particularly important for large-scale transportation applications.

\section{Relationship to Existing Computational Methods and Future Research Directions}

The FTT framework establishes connections with diverse computational traditions in transportation network modeling while opening pathways for future innovations. Our approach builds upon and extends several established methodological streams in traffic assignment, each contributing valuable insights to our integrated tensor-based perspective.

Traditional algorithmic approaches in traffic assignment have evolved through distinct computational paradigms. Path-based methods \citep{jayakrishnan1994faster, xie2018greedy} directly optimize path flows—a concept that naturally aligns with our tensor formulation, as these path flows become tensor slices with explicit gradient relationships. Similarly, origin-based algorithms, column-generation methods \citep{lu2009equivalent} and bush-based approaches \citep{bar2002origin, nie2012note, gentile2014local} decompose problems by origin node, corresponding to our framework's capacity to organize the OD tensor along origin dimensions for efficient parallel processing. Recent advancements include the work of \citep{liu2024inertial}, who proposed a novel inertial-type conjugate gradient projection algorithm with rigorous convergence guarantees.

The critical question of convergence analysis \citep{boyce2004convergence} gains new dimensions in our framework through tensor decomposition techniques already being applied to traffic patterns \citep{chen2019bayesian, sun2016understanding, wang2023transportation}. These decomposition methods identify dominant patterns that guide computational effort toward the most significant components of network flows. For dynamic applications, our approach connects with both the theoretical foundations \citep{szeto2006dynamic, long2019link, peeta2001foundations} and practical implementations \citep{yildirimoglu2014approximating, idoudi2024smart, idoudi2022agent} of DTA by incorporating temporal dimensions as explicit tensor axes. Another promising extension involves harnessing ML to augment DTA, particularly in simulation-based contexts that become computationally intensive at large scale \citep{ameli2020cross}. Recent ML methods can learn network flow patterns directly (e.g., predicting path flows via Transformers \citep{ameli2025machine}), reducing repetitive equilibrium calculations and speeding up convergence. This approach dovetails with advanced parallel strategies and metaheuristic solution methods, opening up new possibilities for real-time, multimodal applications under the FTT framework. 

In broader transportation planning contexts, the FTT framework relates to recent advances in multi-modal coordination \citep{honma2024locational, luo2024innovation, zhang2022integrated}, differentiable optimization techniques \citep{liu2023end, liu2025end}, and simulation-based urban systems \citep{ameli2022evolution}. These connections extend to activity-based modeling approaches \citep{rasouli2014activity} and complex network design problems \citep{farahani2013review}. As transportation planning continues evolving toward integrated frameworks \citep{zhou2009alternative, papageorgiou2020review}, our tensor-based architecture offers analytical tools that bridge traditional modeling boundaries \citep{florian1987efficient, qian2011computing}.

Building on these connections, the following key research directions emerge that leverage the strengths of our framework while addressing remaining challenges:

\textbf{Computational Advances and Scalability:} The tensor-based representation creates opportunities for leveraging modern parallel computing architectures \citep{ameli2020simulation}. Future research should develop specialized parallelization strategies and hardware acceleration using GPUs/TPUs that exploit the mathematical structure of transportation tensors. The integration of automatic differentiation (AD) and computational graph techniques has revolutionized gradient computation, eliminating manual differentiation errors while enhancing numerical stability \citep{Margossian2019, Paszke2017}. These efficiency gains are demonstrated in CG-enabled multinomial probit ICLV models \citep{Ma2022} and multi-day activity chain estimation \citep{Liu2021}. Beyond AD, \citet{Lederrey2021} leverage stochastic gradient descent and adaptive-batch techniques to enhance computational performance for discrete choice models, while \citet{MartinBaos2023b} refine loss functions using penalized maximum likelihood estimation, demonstrating the robustness of this approach for large-scale datasets. Graph-based processing with specialized tensor computation libraries offers promising pathways for handling high-dimensional multimodal systems, while dynamic tensor decomposition methods could significantly reduce computational requirements for real-time applications.

\textbf{Behavioral Representation:} While our current framework establishes the mathematical foundations for network flow modeling, extending it to incorporate more realistic behavioral elements remains an important direction \citep{vanCranenburgh2022}. Recent studies have revealed significant overlaps between behavioral econometric models and tensor-based machine learning frameworks, particularly in parameter estimation processes where both domains employ non-convex optimization. Hybrid approaches integrate tensor-based ML architectures into econometric models through two primary pathways: First, in model structure enhancement, \citet{Sifringer2020} propose Dense Neural Network (DNN)-embedded utility specifications within MNL frameworks, while \citet{Wang2020a} develop the Alternative-Specific Utility Deep Neural Network to incorporate domain-knowledge and preserve behavioral interpretability. \citet{Han2022} implement DNNs to estimate flexible taste parameters rather than deterministic utilities, enabling complex preference variations. \citet{Phan2022} further incorporate attention mechanisms to dynamically quantify taste heterogeneity, and \citet{MartinBaos2024} develop kernel logistic regression (KLR) embedded models to improve transport demand analysis. Such enhancements not only improve predictive accuracy but also provide key economic indicators similar to traditional econometric models, such as market shares, Willingness to Pay and Value of Time \citep{MartinBaos2023a, Wang2020b}. Future research should integrate multimodal and multisource data to enhance the modernization and real-time adaptability of activity-based models and behavioral analysis.

\textbf{Policy Applications:} The analytical power of the tensor-based approach enables quantitative assessment of complex multimodal network impacts, with direct applications to congestion pricing, parking management, and infrastructure investment decisions. Distributed implementations across operational entities would facilitate coordination in environments with fragmented governance structures, while real-time control applications would support sustainable urban mobility through rigorous multimodal transportation coordination. Our framework also connects with emerging research on system resilience, such as resilience as a service frameworks \citep{amghar2024resilience, jaber2025methodological} and spatiotemporal models for railway timetable resilience \citep{chen2025enhancing}, which emphasize the importance of robust transportation systems under disturbance scenarios.

By building upon established computational methods while advancing these research directions, the FTT framework provides both immediate practical contributions and a foundation for future innovations in transportation network modeling and optimization.

\section{Acknowledgements}

M. Ameli acknowledges support from the French ANR research project SMART-ROUTE (grant number ANR-24-CE22-7264).

Zhou and Zhu are supported by the National Science Foundation (NSF) under grant no. TIP-2303748, titled ”CONNECT: Consortium of Open-Source Planning Models for Next-Generation Equitable and Efficient Communities and Transportation”.
\section*{Conflicts of interest}
None.

\setcounter{equation}{0}
\renewcommand{\theequation}{A.\arabic{equation}}


\appendix
\section{Notation System for Flow-Through Tensor Framework} \label{appendixa}
\addcontentsline{toc}{section}{Notation System for Flow-Through Tensor Framework}
\label{appendixC}
\vspace{-15pt} 
\begin{table}[H]
\centering
\caption{Comprehensive Notation in the Flow-Through Tensor Framework}
\renewcommand{\arraystretch}{1.2}
\begin{tabular}{>{\bfseries}l p{10cm}}
\toprule
\multicolumn{2}{l}{\textit{Sets and Indices}} \\
\midrule
$O$ & Set of origin nodes $o$\\
$D$ & Set of destination nodes $d$\\
$OD$ & Set of origin-destination pairs $(o,d)$ \\
$P$ & Set of paths $p$ \\
$P_{(o,d)}$ & Set of paths connecting OD pair $(o,d)$ \\
$L$ & Set of links $\ell$ in the network \\
\midrule
\multicolumn{2}{l}{\textit{Flow Variables}} \\
\midrule
$f_{od}$ & Flow between origin $o$ and destination $d$ \\
$f_p$ & Flow on path $p$ \\
$f_\ell$ & Flow on link $\ell$ \\
$\mathbf{f}_{OD}$ & Tensor form of flows between OD pairs \\
$\mathbf{f}_P$ & Tensor form of path flows \\
$\mathbf{f}_L$ & Tensor form of link flows \\
$\mathcal{F}_{OD}$ & Tensor of vehicle OD flows \\
$\mathcal{F}_{P}$ & Tensor of vehicle path flows \\
$\mathcal{F}_{L}$ & Tensor of vehicle link flows \\
\midrule
\multicolumn{2}{l}{\textit{Travel Time Variables}} \\
\midrule
$t_{od}$ & Travel time between origin $o$ and destination $d$ \\
$t_p$ & Travel time on path $p$ \\
$t_\ell$ & Travel time on link $\ell$ \\
$t_\ell^0$ & Free-flow travel time on link $\ell$ \\
$\mathbf{t}_{OD}$ & Vector form of OD travel times \\
$\mathbf{t}_P$ & Vector form of path travel times \\
$\mathbf{t}_L$ & Vector form of link travel times \\
$\mathcal{T}_{OD}$ & Tensor of vehicle OD travel times \\
$\mathcal{T}_{P}$ & Tensor of vehicle path travel times  \\
$\mathcal{T}_{L}$ & Tensor of vehicle link travel times  \\
\midrule
\multicolumn{2}{l}{\textit{Incidence Matrices}} \\
\midrule
$\mathbf{B}_{\mathrm{OD},\mathrm{P}}$ & OD-to-Path probability matrix where $b_{od,p}$ is the probability of an agent from OD pair $(o,d)$ choosing path $p$; dimension: $\lvert OD\rvert \times \lvert P\rvert$ \\
$\mathbf{A}_{\mathrm{P},\mathrm{L}}$ & Path-to-Link incidence matrix where $a_{p,\ell} = 1$ if link $\ell$ is on path $p$; dimension: $\lvert P\rvert \times \lvert L\rvert$ \\
\bottomrule
\end{tabular}
\label{tab:appendix_notation}
\end{table}

\section{Backpropagation and the Method of Adjoints in Flow-Through-Tensor Networks}
\label{appendixb}

The method of adjoints, pioneered by \cite{bryson1961gradient} and elaborated by \cite{BrysonHo69}, provides a framework that directly parallels backpropagation in neural networks. This connection reveals how neural network training fundamentally solves an optimal control problem through Lagrangian duality.

\subsection*{Neural Networks and Flow-Through-Tensor Framework as Constrained Optimization}

Standard deep learning formulates a composition of functions $\psi(x;\vartheta)=f_{\ell}\circ f_{\ell-1}\circ\cdots\circ f_1(x)$ to minimize:
\begin{equation}
\min_{\psi}\frac{1}{n}\sum_{k=1}^{n}\text{loss}(\psi(x_k),y_k). \tag{B.1}
\end{equation}

This can be reformulated as a constrained problem:
\begin{align}
\min_{\vartheta}&\; \frac{1}{n}\sum_{k=1}^{n}\text{loss}(z_k^{(\ell)},y_k) \tag{B.2a} \\
\text{subject to: } &z_k^{(i)}=f_i(z_k^{(i-1)},\vartheta_i) \text{ for } i=1,2,...,\ell, \tag{B.2b}
\end{align}
 In this context, \(y_k\) represents the desired or target output for the \(k\)th example. \(\vartheta = \{\vartheta_1, \vartheta_2, \dots, \vartheta_\ell\}\) is the set of parameters
   to be learned; these affect how the input is transformed through the layers.

The corresponding Lagrangian for this problem is:
\begin{equation}
\mathcal{L}(z,\vartheta,\mathbf{p}):=\text{loss}(z^{(\ell)},y)-\sum_{i=1}^{\ell}\mathbf{p}_i^{\mathsf{T}}(z^{(i)}-f_i(z^{(i-1)},\vartheta_i)). \tag{B.3}
\end{equation}

Taking derivatives with respect to each layer's activations and parameters:
\begin{align}
\nabla_{z^{(i)}}\mathcal{L} &= -\mathbf{p}_i+\nabla_{z^{(i)}}f_{i+1}(z^{(i)},\vartheta_{i+1})^{\mathsf{T}}\mathbf{p}_{i+1}, \tag{B.4a} \\
\nabla_{z^{(\ell)}}\mathcal{L} &= -\mathbf{p}_{\ell}+\nabla_{z^{(\ell)}}\text{loss}(z^{(\ell)},y), \tag{B.4b} \\
\nabla_{\vartheta_i}\mathcal{L} &= \nabla_{\vartheta_i}f_i(z^{(i-1)},\vartheta_i)^{\mathsf{T}}\mathbf{p}_i. \tag{B.4c}
\end{align}

Setting these derivatives to zero reveals that the adjoint variables $\mathbf{p}_i$ propagate gradients backward through the network, with $\mathbf{p}_\ell$ initialized at the final layer as the gradient of the loss function, and subsequent adjoints computed recursively backward through the network.


Our Flow-Through-Tensor framework follows an analogous structure. For a transportation network with fixed OD flows $\mathbf{f}_{OD}$:

\begin{align}
\min_{\mathbf{f}_P} &\; Z(\mathbf{f}_L, \mathbf{t}_L, \mathbf{t}_P). \tag{B.5a} \\
\text{subject to: } &\mathbf{B}^{I}\mathbf{f}_P = \mathbf{f}_{OD} \text{ (flow conservation)} \tag{B.5b} \\
&\mathbf{f}_P \geq 0 \text{ (non-negativity)} \tag{B.5c} \\
&\mathbf{f}_L = \mathbf{A}^{\mathsf{T}}\mathbf{f}_P \text{ (link flows)} \tag{B.5d} \\
&\mathbf{t}_L = \phi(\mathbf{f}_L) \text{ (link travel times)} \tag{B.5e} \\
&\mathbf{t}_P = \mathbf{A}\mathbf{t}_L \text{ (path travel times)}, \tag{B.5f}
\end{align}

While the neural network explicitly optimizes $\vartheta$, the FTT model primarily focuses on optimizing variables such as the path flows $\mathbf{f}_P$. The structure of the FTT model is dictated by fixed components:
\begin{itemize}
    \item \textbf{Matrices:} Matrices such as $\mathbf{A}$ and $\mathbf{B}^{I}$ define relationships (e.g., link flows and flow conservation).
    \item \textbf{Nonlinear Functions:} Nonlinear functions like $\phi(\cdot)$ capture complex relationships such as volume delay.
\end{itemize}
Thus, while the primary structural parameters (e.g., the matrices and the functional form of $\phi$) are typically fixed or pre-calibrated, the variables $\mathbf{f}_L$, $\mathbf{t}_L$, $\mathbf{t}_P$, and even $\mathbf{f}_{OD}$ (if not strictly given) are computed through the forward pass and updated during the optimization process. In this sense, these variables are "learnable" in that their values are adjusted as part of solving the overall constrained optimization problem—much as the activations in a neural network are computed and then used to determine gradients for updating $\vartheta$.

The Lagrangian for this transportation problem is:
\begin{align}
\mathcal{L} = &Z(\mathbf{f}_L, \mathbf{t}_L, \mathbf{t}_P) - \mathbf{p}_1^{\mathsf{T}}(\mathbf{f}_L - \mathbf{A}^{\mathsf{T}} \cdot \mathbf{f}_P) - \mathbf{p}_2^{\mathsf{T}}(\mathbf{t}_L - \phi(\mathbf{f}_L)) \nonumber \\
&- \mathbf{p}_3^{\mathsf{T}}(\mathbf{t}_P - \mathbf{A}^{\mathsf{T}}\cdot\mathbf{t}_L) - \boldsymbol{\lambda}^{\mathsf{T}}(\mathbf{B}^{I}\cdot\mathbf{f}_P - \mathbf{f}_{OD})  +\boldsymbol{\mu}^{\mathsf{T}}\mathbf{f}_P \tag{B.6}
\end{align}
where $\boldsymbol{\mu} \leq 0$, the vectors $\mathbf{p}_1$, $\mathbf{p}_2$, and $\mathbf{p}_3$ are adjoint variables that correspond to the constraints that link variables across layers. The Karush-Kuhn-Tucker (KKT) conditions yield the following.

\begin{align}
\nabla_{\mathbf{f}_L}\mathcal{L} &= \nabla_{\mathbf{f}_L}Z - \mathbf{p}_1 - \nabla_{\mathbf{f}_L}\phi(\mathbf{f}_L)^{\mathsf{T}} \mathbf{p}_2= 0, \tag{B.7a} \\
\nabla_{\mathbf{t}_L}\mathcal{L} &= \nabla_{\mathbf{t}_L}Z - \mathbf{p}_2 - \mathbf{A}^{\mathsf{T}}\mathbf{p}_3 = 0, \tag{B.7b} \\
\nabla_{\mathbf{t}_P}\mathcal{L} &= \nabla_{\mathbf{t}_P}Z - \mathbf{p}_3 = 0, \tag{B.7c} \\
\nabla_{\mathbf{f}_P}\mathcal{L} &= \mathbf{A}\mathbf{p}_1 - (\mathbf{B}^{\mathrm{I}})^{\mathsf{T}}\boldsymbol{\lambda} + \boldsymbol{\mu} = 0. \tag{B.7d}
\end{align}

\subsection*{Solution Process and Chain Rule Structure}

The solution process involves two passes analogous to neural network training. The forward pass computes all variables through the network:
\begin{align}
\mathbf{f}_L &= \mathbf{A}\mathbf{f}_P \tag{B.8a} \\
\mathbf{t}_L &= \phi(\mathbf{f}_L) \tag{B.8b} \\
\mathbf{t}_P &= \mathbf{A}\mathbf{t}_L \tag{B.8c}
\end{align}

The backward pass solves for adjoint variables recursively, starting from the output layer and moving backward:
\begin{align}
\mathbf{p}_3 &= \nabla_{\mathbf{t}_P}Z \tag{B.9a} \\
\mathbf{p}_2 &= \nabla_{\mathbf{t}_L}Z + \mathbf{A}\mathbf{p}_3 \tag{B.9b} \\
\mathbf{p}_1 &= \nabla_{\mathbf{f}_L}Z + \nabla_{\mathbf{f}_L}\phi(\mathbf{f}_L)^{\mathsf{T}}\mathbf{p}_2 \tag{B.9c}
\end{align}

These adjoint variables represent sensitivities that propagate backward through the computational graph. The final gradient with respect to path flows is:
\begin{equation}
\nabla_{\mathbf{f}_P}Z = \mathbf{A}\mathbf{p}_1 - (\mathbf{B}^{\mathrm{I}})^{\mathsf{T}} \boldsymbol{\lambda} + \boldsymbol{\mu}
 \tag{B.10}
\end{equation}

The nested dependencies create a chain rule structure identical to neural network backpropagation:
\begin{equation}
\frac{\partial Z}{\partial \mathbf{f}_P} = \mathbf{A}\left[\frac{\partial Z}{\partial \mathbf{f}_L} + \frac{\partial \phi}{\partial \mathbf{f}_L}^{\mathsf{T}}\left(\frac{\partial Z}{\partial \mathbf{t}_L} + \mathbf{A}^{\mathsf{T}}\frac{\partial Z}{\partial \mathbf{t}_P}\right)\right] \tag{B.11}
\end{equation}

\section{Proof of Pareto-Improving Coordination Benefits}
\label{appendixc}

\setcounter{equation}{0}
\renewcommand{\theequation}{B.\arabic{equation}}

\begin{proof}[Proof of Theorem~\ref{thm:pareto}]
Under UE, the system cost is $t_{UE} = 1$. With rotation, the system cost becomes:
\begin{align}
    t_{SO} &= p \cdot \overline{t}_{P} + (1 - p) \cdot \overline{t}_{NP} \notag \\
                   &= \frac{p}{2} + \left(1 - \frac{p}{2} \right)^{\beta + 1} \tag{C.1}\label{eq:system-cost}
\end{align}

Under the pure user equilibrium (i.e., when $p = 0$), every traveler uses Route $b$, so:
\begin{equation}
    t_{UE} = 1^{\beta} = 1 \tag{C.2}\label{eq:UE-cost}
\end{equation}

Therefore, the road user benefit from the rotation $\Delta_p$ is the difference between the user equilibrium cost and the system cost:
\begin{align}
\Delta_p &= t_{UE} - t_{SO} \notag \\
         &= 1 - \left[ \frac{p}{2} + \left(1 - \frac{p}{2} \right)^{\beta + 1} \right] \tag{C.3}\label{eq:benefit}
\end{align}

To approximate the benefit for small $p$, we apply a first-order Taylor approximation to the nonlinear term:
\begin{equation}
\left(1 - \frac{p}{2} \right)^{\beta+1} \approx 1 - \frac{(\beta+1) p}{2} \tag{C.4}
\end{equation}

Thus, the approximate benefit is:
\begin{equation}
\Delta_p \approx \frac{\beta p}{2} \tag{C.5}
\end{equation}

This confirms that the benefit is approximately linear in both the congestion sensitivity $\beta$ and the participation rate $p$.
\end{proof}
\vspace{-1em}

\bibliographystyle{elsarticle-harv}
\bibliography{cas-refs}

\end{document}